\documentclass[11pt]{amsart}

\usepackage{lineno}

\usepackage[numbers,sort]{natbib}

\usepackage{amsfonts}
\usepackage{amsmath}
\usepackage{amssymb}
\usepackage{amscd}
\usepackage{color}
\usepackage{xcolor}
\usepackage{amsthm}
\usepackage[hmargin=4cm,vmargin=4cm]{geometry}
\usepackage{parskip}

\usepackage{enumitem}

\usepackage{eucal}

\newtheorem{thm}{Theorem}

\newtheorem{thm*}{Theorem}
\newtheorem{prop}{Proposition}
\newtheorem{lma}[prop]{Lemma}

\newtheorem{cor}[prop]{Corollary}

\theoremstyle{definition}

\newtheorem{df}[prop]{Definition} 

\theoremstyle{remark}

\newtheorem{rmk}[prop]{Remark} 

\newtheorem{qtn}[prop]{Question}

\newcommand{\R}{{\mathbb{R}}}
\newcommand{\Z}{{\mathbb{Z}}}
\newcommand{\C}{{\mathbb{C}}}
\newcommand{\Q}{{\mathbb{Q}}}

\newcommand{\D}{{\mathbb{D}}}
\newcommand{\bK}{{\mathbb{K}}}

\let\Im\relax
\DeclareMathOperator{\Im}{\mathrm{Im}}

\newcommand{\del}{\partial}

\newcommand{\sm}[1]{C^\infty(#1)}
\newcommand{\smc}[1]{C^\infty_c(#1)}
\renewcommand{\dh}[2]{d_{\mrm{Hofer}}(#1,#2)}
\newcommand{\vareps}[1]{\varepsilon_{#1}}

\newcommand\vol{\operatorname{vol}}

\newcommand{\cL}{\mathcal{L}}

\newcommand{\intoi}{\int_0^1}
\newcommand{\til}[1]{\widetilde{#1}}

\newcommand{\ol}[1]{\overline{#1}}
\newcommand{\arr}[1]{\overrightarrow{#1}}

\newcommand{\blue}[1]{{\color{black} #1}}

\newcommand{\zt}{{\Z/2\Z}}

\DeclareMathOperator{\ima}{\mathrm{im}}

\newcommand{\om}{\omega}

\newcommand{\eps}{\epsilon}

\newcommand{\cA}{\mathcal{A}}

\newcommand{\cC}{\mathcal{C}}

\newcommand{\cE}{\mathcal{E}}
\newcommand{\cF}{\mathcal{F}}

\newcommand{\cI}{\mathcal{I}}

\newcommand{\cR}{\mathcal{R}}

\newcommand{\LH}{{H,L}}

\def\mrm#1{{\mathrm{#1}}}
\def\bb#1{{\mathbb{#1}}}
\def\cl#1{{\mathcal{#1}}}
\def\bf#1{{\mathbf{#1}}}

\newcommand{\brat}[1]{{\left< #1 \right>}}

\newcommand{\bfb}{{\bf{b}, b}}

\DeclareMathOperator{\Sym}{\mathrm{Sym}}

\DeclareMathOperator{\Ham}{\mathrm{Ham}}

\DeclareMathOperator{\Symp}{\mathrm{Symp}}
\DeclareMathOperator{\Ker}{\mathrm{Ker}}

\DeclareMathOperator{\Cal}{\mathrm{Cal}}

\DeclareMathOperator{\id}{\mathrm{id}}
\DeclareMathOperator{\spec}{\mathrm{Spec}}

\def\H2{H^{(2)}}

\begin{document}

\title[Lagrangian configurations and Hamiltonian maps]{Lagrangian configurations \\ and Hamiltonian maps}

\author{Leonid Polterovich and Egor Shelukhin}
\date{}

\begin{abstract}
We study configurations of disjoint Lagrangian submanifolds in certain low-dimensional symplectic manifolds from the perspective of the geometry of Hamiltonian maps. We detect infinite-dimensional flats in the Hamiltonian group of the two-sphere equipped with Hofer's metric, prove constraints on Lagrangian packing, find instances of Lagrangian Poincar\'{e} recurrence, and present a new hierarchy of normal subgroups of area-preserving homeomorphisms of the two-sphere.   The technology involves Lagrangian spectral invariants with Hamiltonian term in symmetric product orbifolds.
\end{abstract}


\maketitle

\tableofcontents

\section{Introduction and main results}

\subsection{Overview}

A symplectic structure $\omega$ on an even-dimensional manifold $M^{2n}$ can be viewed as
a far reaching generalization of the two-dimensional area on surfaces: by definition,
$\omega$ is a closed differential $2$-form whose top wedge power $\omega^n$ is a volume form.
The group $\text{Symp}(M,\omega)$ of all symmetries of a symplectic manifold, i.e. of diffeomorphisms preserving the symplectic structure, contains a remarkable subgroup of {\it Hamiltonian diffeomorphisms} $\Ham(M,\omega)$. When $M$ is closed and its first cohomology vanishes, $\Ham$ coincides with the identity component of $\text{Symp}$. At the same time,
in classical mechanics, where $M$ models the phase space, $\Ham$ arises as the group of all
admissible motions. This group  became a central object
of interest in symplectic topology in the past three decades. In spite of that, some very basic
questions about the algebra, geometry and topology of $\Ham(M,\omega)$ are far from being understood
even when $M$ is a surface.

Let us briefly outline the contents of the paper.
First, we obtain new results on Hofer's bi-invariant geometry of $\Ham(M,\omega)$ in the case where $M=S^2$ is the two-dimensional sphere.
We establish, roughly speaking, that $\Ham(S^2)$ contains flats of arbitrary dimension, thus
solving a question open since 2006. A similar result was obtained simultaneously and independently, by using a different technique, in a recent paper by Cristofaro-Gardiner, Humili\`{e}re and Seyfaddini \cite{CGHS2}. Furthermore, our method yields an infinite hierarchy of normal subgroups of area-preserving homeomorphisms of the two-sphere, all of which contain the normal subgroup of homeomorphisms of finite energy discovered in \cite{CGHS2}. Additionally, we find a new constraint on rotationally symmetric homeomorphisms of finite energy. 

Second, we extend a number of elementary two-dimensional phenomena taking place on $S^2$ to
the {\it stabilized} space $S^2 \times S^2$, where the area of the second factor is ``much smaller" than that of the first one. These phenomena include constraints on packing by circles, which correspond to packings by two-dimensional tori after stabilization, and a  version of Poincar\'{e} recurrence theorem for sets of zero measure in the context of Hamiltonian diffeomorphisms. Furthermore, we prove a stabilized version of our results on Hofer's geometry presented above.

To illustrate the stabilization paradigm, an elementary area count shows that one cannot fit into the sphere of unit area  $k$ pair-wise disjoint Hamiltonian images of a circle $L$ bounding a disc of area $> 1/k$. We shall show that the same is true for $L \times \text{equator}$ in the stabilized space. Here the area/volume control miserably fails: a two-dimensional torus does not bound any
volume in a four-manifold!

The first instance of such a stabilization was discovered in a recent paper by Mak and Smith \cite{MakSmith-links}, who studied constraints on the displaceability of collections of curves on $S^2$.
Recall, that a set $X \subset M$ is called displaceable if there exists a Hamiltonian diffeomorphism $\phi$ with
$\phi(X) \cap X = \emptyset$. This notion, introduced by Hofer in 1990, gives rise to a natural ``small scale" on symplectic manifolds.
Mak and Smith noticed that certain ``small sets" on $M$ become more rigid when one looks at them in the symmetric product, a {\it symplectic orbifold}  $(M)^k/\Sym_k$ where $\Sym_k$ stands for the permutation group. On the technical side, our main observation is that a powerful Floer-theoretical tool, Lagrangian spectral invariants with Hamiltonian term, as developed by Leclercq-Zapolsky and Fukaya-Oh-Ohta-Ono, extends to Lagrangian tori in symmetric product orbifolds and can be applied to the study of the above-mentioned questions on  Hamiltonian maps. With this language, our paper provides a toolbox for {\em measurements of large energy symplectic effects on
small geometric scales} by using Floer theory in symmetric products.

Let us note that the idea of looking at configuration spaces of points on a surface, the two-sphere in particular, in order to construct invariants of Hamiltonian diffeomorphisms goes back to Gambaudo and Ghys \cite{gambaudo2004commutators}. Technically, it turns out that for this paper it is beneficial to work with symmetric products, which are certain compactifications of unordered configuration spaces, and instead of the sphere to look at a certain stabilization of it to a four-manifold. Thus our approach can be considered as a ``symplectization" of the one by Gambaudo and Ghys. Furthermore, the central objects of
interest to us are certain collections of pair-wise disjoint Lagrangian submanifolds
which in some sense govern Hamiltonian dynamics. These are
{\it the Lagrangian configurations} appearing in the title of the paper.

\subsection{Flats in Hofer's geometry}

In \cite{HoferMetric} Hofer has introduced a remarkable bi-invariant metric on the group $\Ham(M,\om)$ of Hamiltonian diffeomorphisms of a symplectic manifold $(M,\om).$ It can intuitively be thought of as the minimal $L^{\infty,1}$ energy required to generate a given Hamiltonian diffeomorphism. For a time-dependent Hamiltonian $H$ in $\sm{[0,1]\times M, \R}$ we denote by $\{ \phi^t_H \}_{t \in [0,1]}$ the Hamiltonian isotopy generated by the vector field $X^t_H$ given by \[\om(X^t_H, \cdot) = - dH_t(\cdot),\] where $H_t(\cdot) = H(t,\cdot)$ for all $t \in [0,1].$ We say that $H$ has zero mean if  $\int H_t\, \om^n = 0$ for all $t \in [0,1].$ When the symplectic manifold is closed, the Hofer distance of $\phi \in \Ham(M,\om)$ to the identity is defined as
\[d_{\rm{Hofer}}(\phi, \id) = \inf_{\phi_H^1 = \phi} \intoi \max_M |H(t,-)|\; dt,\]
where the infimum is taken over all zero mean Hamiltonians generating $\phi$.
It is extended to arbitary pairs of diffeomorphisms by the bi-invariance, \[d_{\rm{Hofer}}(\psi \phi, \psi \phi') = d_{\rm{Hofer}}( \phi \psi, \phi' \psi) = d_{\rm{Hofer}}(\phi, \phi'),\] for all $\phi,\phi',\psi \in \Ham(M,\om).$

The non-degeneracy of $d_{\mrm{Hofer}}$ was studied in many further publications such as \cite{Viterbo-specGF,Polterovich-displ} and was proven to hold for arbitrary symplectic manifolds in \cite{Lalonde-McDuff-Energy}. Let us mention also that Hofer's metric naturally lifts to a (pseudo)-metric on the universal cover $\til{\Ham}(M,\om)$, where the question about its non-degeneracy is still open in full generality. We refer to \cite{P-book} for a detailed introduction to the Hofer metric and many of its aspects and properties.

The main question related to the Hofer metric, Problem 20 in \cite{McDuffSalamonIntro3}, is whether its diameter is infinite for all symplectic manifolds, and when it is, which unbounded groups can be quasi-isometrically embedded into $(\Ham(M,\om),d_{\mrm{Hofer}}).$

\begin{thm}\label{thm: main-L} The additive group $\cl{G} = C^{\infty}_c(I)$ of compactly supported smooth functions on an open interval $I$ with the $C^0$ distance embeds isometrically into $\Ham(S^2)$ endowed with $d_{\mrm{Hofer}}$.
\end{thm}

This settles a question of the first author and Kapovich from $2006$, cf. Problem 21  \cite{McDuffSalamonIntro3}. A similar result was obtained simultaneously and independently in \cite{CGHS2} by completely different methods based on periodic Floer homology.
As, by a classical theorem \cite[Th\'eor\`eme 10, p. 187]{Banach} from a book of Banach, every separable metric space admits an isometric embedding into $C^0_c(I),$ and by smooth approximation, the latter is quasi-isometric to $C^{\infty}_c(I),$ Theorem \ref{thm: main-L} implies the following.

\begin{cor}\label{cor: Banach}
Every separable metric space admits a quasi-isometric embedding into $\Ham(S^2)$ endowed with $d_{\mrm{Hofer}}$.
\end{cor}

For closed surfaces of higher genus, flats of arbitrary dimension were found in \cite{Py-plats}. The proof is based on the fact that such a surface admits an
incompressible annulus foliated by non-displaceable closed curves. Symplectic rigidity
of these curves yields the result. The lack of such an annulus in the case of $S^2$ requires a new tool, {\it Lagrangian estimators}, which we develop by using Lagrangian Floer theory in orbifolds, see Sections \ref{sec-estimators} and \ref{sec: prelim}.

In fact,  our method of proof yields the following statement about Hofer's geometry in dimension four. Throughout the paper $S^2(b)$ stands for the sphere equipped with the standard area form normalized in such a way that the total area equals $b$. We abbreviate $S^2=S^2(1)$.

For $a > 0$, consider a symplectic manifold $M_a = S^2 \times S^2(2a)$.
The natural monomorphism \[\iota: {\Ham}(S^2) \to {\Ham}(M_a),\;\; \phi \mapsto \phi \times \id,\] satisfies $d_{\mrm{Hofer}}(\iota(\phi), \iota(\psi)) \leq d_{\mrm{Hofer}}(\phi, \psi)$.

\begin{thm}\label{thm: main2-L} Assume that $a>0$ is small enough. The isometric
monomorphism $\Phi: \cl{G}\to \Ham(S^2)$ from Theorem \ref{thm: main-L}
can be chosen in such a way that
\[\Psi = \iota \circ \Phi: (\cl{G}, d_{C^0}) \hookrightarrow (\Ham(M_a),d_{\rm{Hofer}})\] is a bi-Lipschitz embedding. Furthermore,  the monomorphism of the universal covers $\til{\Psi}:\cl{G} \to \til{\Ham}(M_a)$ covering $\Psi$ is an isometric embedding.
\end{thm}


\subsection{Lagrangian packing}\label{subsec-lp}
Let $K_r$ be a simple closed curve on the sphere $S^2=S^2(1)$
bounding a disc of area $1/{k} > r > 1/(k+1)$, $k \in \mathbb{N},$ $k\geq 2$. Let
$S \subset S^2(2a)$ be the equator. Note that the Lagrangian torus
$\Lambda_r = K_r \times S$ can be Hamiltonianly $k$-packed into $M_a$, that is, there are $k$ Hamiltonian diffeomorphisms $\phi_1 = \id, \phi_2,\ldots,\phi_k$ of $M_a$ such that $\{\phi_j (\Lambda_{r}) \}_{1 \leq j \leq k}$ are pairwise disjoint. Indeed, such a packing exists for
$K_r \subset S^2$.

\begin{thm}[Lagrangian packing]\label{cor-lp} One cannot pack $M_a = S^2 \times S^2(2a)$ by $k+1$ Hamiltonian images of $\Lambda_r$ if $a$ is sufficiently small.
\end{thm}

We present an argument, based on asymptotic Hofer's geometry, in Section \ref{sec-lp}.

\subsection{Lagrangian Poincar\'{e} recurrence} Using the Lagrangian packing obstructions, we are able to make the following progress on the well-known Lagrangian Poincar\'{e} recurrence conjecture in Hamiltonian dynamics \cite{GG-Arnold}.
It is a version of the classical Poincar\'{e} recurrence theorem, but in the setting of Lagrangian submanifolds instead of sets of positive measure. Note that except for the case of surfaces, this is a purely symplectic question, since Lagrangian submanifolds do not bound and they have zero measure.

Let $\Lambda_r \subset M_a$ be a Lagrangian torus as in Section \ref{subsec-lp}.
For an arbitrary Hamiltonian diffeomorphism $\phi \in \Ham(M_a)$, consider the {\it recurrence set}
$$\mathcal{R}_\phi := \{n \in \mathbb{N} \;:\; \phi^n \Lambda_r \cap \Lambda_r \neq \emptyset\}\;.$$

\begin{thm}[Lagrangian recurrence]\label{cor-lr} The lower density of $\cR_\phi$ is at least
	$1/k$.
\end{thm}

We remark that while in \cite{GG-pseudorotations} a similar statement was proven for specific rigid Hamiltonian diffeomorphisms of complex projective spaces (pseudo-rotations) and arbitrary Lagrangians, we provide the first non-trivial higher-dimensional examples of Lagrangian submanifolds satisfying the recurrence property for {\em all} Hamiltonian diffeomorphisms.

\subsection{Area-preserving homeomorphisms of $S^2$}
It has been established by Cristofaro-Gardiner, Humili\`{e}re and Seyfaddini \cite{CGHS2}
that the group $G_{S^2}$ of symplectic homeomorphisms of the sphere possesses a non-trivial normal subgroup $G_{S^2}^{F}$ of homeomorphisms of finite energy. These are the homeomorphisms $\phi$ for which there exists a sequence of Hamiltonian diffeomorphisms $\psi_j$ and a constant $C > 0$ such that $d_{C^0}(\psi_j,\phi) \to 0$ and $d_{\mrm{Hofer}}(\psi_j,\id) \leq C.$ This is achieved by showing that radially symmetric Hamiltonian homeomorphisms of bounded energy satisfying a certain monotone twist condition must, in a precise sense, have finite Calabi invariant. We extend this result as follows.

Let $z:S^2 \to [-1/2,1/2]$ be the moment map for the standard $S^1$-action on $S^2.$ It is the height function for the standard embedding of $S^2$ in $\R^3$ scaled by a factor of $1/2.$ Let $h:[-1/2,1/2) \to \R$ be a smooth function that vanishes on $[-1/2,0].$ Consider $\phi \in G_{S^2}$ generated by $H = h \circ z:$ it is the $C^0$-limit of $\phi_i = \phi^1_{H_i} \in \Ham_c (\bb D^2) \subset \Ham(S^2)$ for $H_i = h_i \circ z$ for $h_i$ approximations to $h$ constant near $1/2$ which agree with $h$ on $[-1/2,1/2-1/i].$

\begin{thm}\label{thm: non-simplicity}
If $\phi \in G_{S^2}^F$	then the primitive of $h$ is a bounded function.
\end{thm}


The proof is based on a combination of \cite[Lemma 3.11]{CGHS2}, an interesting soft result relating $C^0$-smallness, supports, and the Hofer metric, with our Lagrangian
spectral estimators. It turns out that suitable linear combination of these invariants are continuous with respect to the $C^0$ topology on $\Ham(S^2)$ and extend to the group $G_{S^2}$.
In fact, our invariants give rise to an infinite series of pair-wise distinct normal subgroups
of  $G_{S^2}$ containing $G^F_{S^2}$, see Section \ref{sec: C0} for a precise formulation.


\begin{rmk}\label{rmk: discrepancy}
It is easy to see that if $h$ is bounded then $\phi \in G^F_{S^2}.$ It would be very interesting to bridge the discrepancy between this observation and Theorem \ref{thm: non-simplicity} further. Note that there do exist unbounded functions $h$ whose primitives are bounded. We expect that arguments along the lines of Sikorav's trick as in Section \ref{sec-lp} and lower bounds in terms of further linear combinations of the $c_{k,B}$ invariants could be useful to this end.
\end{rmk}

\begin{rmk}
We note that a suitable generalized limit construction (e.g., the Banach limit or the limit with respect to a non-principal ultrafilter) yields the existence of many homomorphisms $\phi \mapsto \mathcal{C}(\phi) \in \R$ on the group of $\phi$ as in Theorem \ref{thm: non-simplicity} that coincide with the Calabi invariant of $\phi \in \Ham_c(\D^2)$ if $H = h \circ z$ extends smoothly to $S^2.$
\end{rmk}


\color{black}

\medskip
\noindent
{\sc Organization of the paper:} In Section \ref{sec-estimators} we introduce and list
the properties of Lagrangian estimators, our main technical tool. The detailed construction
of the estimators via Lagrangian Floer theory in symmetric products, as well as the proof of their properties is presented in Section \ref{sec: prelim}.

Section \ref{sec-H} deals with flats in Hofer geometry. We prove Theorems \ref{thm: main-L} and \ref{thm: main2-L}, and present a result on infinite-dimensional flats in subgroups of Hamiltonian diffeomorphisms of the disc with vanishing Calabi invariant. Additionally,
we derive an estimate on the asymptotic Hofer norm which is used in the proof
of Theorem \ref{cor-lp} on Lagrangian packing. This proof can be found in Section \ref{sec-lp}.

In Section \ref{sec-lr} we deduce Theorem \ref{cor-lr} on Lagrangian recurrence
from the packing constraint by a combinatorial argument.

 In Section \ref{sec: C0} we prove the $C^0$-continuity of certain linear combinations of our invariants and present applications to normal subgroups of the group of area-preserving homeomorphisms of $S^2$.
\color{black}

We conclude the paper with a discussion of further directions in Section \ref{sec-discussion}.

\section{Lagrangian estimators}\label{sec-estimators}

In this section we introduce three slightly different
flavors of {\it Lagrangian estimators}, the main technical tool
of the present paper. These are functionals with a number of remarkable properties defined on the space of time-dependent Hamiltonians {\it (spectral estimators)}, on the group of Hamiltonian diffeomorphisms {\it (group estimators)}, and on the Lie algebra of functions on a symplectic manifold {\it (algebra estimators)}. They somewhat resemble, but are different from, partial symplectic quasi-morphisms and quasi-states introduced in \cite{EntovPolterovich-intersections}, respectively (see Remark \ref{rem-qs} below).

We rely on the theory of Lagrangian spectral invariants \cite{Leclercq-spectral,LeclercqZapolsky,FO3-qm} in the setting of bulk-deformed Lagrangian Floer homology \cite{FO3:book-vol12} for symplectic orbifolds \cite{ChoPoddar} and crucially its recent investigation \cite{MakSmith-links} in the context of Lagrangian links in symplectic four-manifolds. The output of our construction is a new invariant of a Hamiltonian diffeomorphism of $S^2$ that instead of being supported on a single non-displaceable Lagrangian circle is supported on a non-displaceable
configuration of pair-wise disjoint and (in general) individually displaceable
Lagrangian circles.

We start with a couple of preliminary notions and notations. Let $z:S^2 \to [-1/2,1/2]$ be the moment map for the standard $S^1= \R/\Z$-action on $S^2$. It is instructive to think that $S^2=S^2(1)$ is the round sphere of radius $1/2$ in $\R^3$ equipped with the standard area form divided by $\pi$, and $z$ is simply the vertical Euclidean coordinate.

Let $k\geq 1$ be a positive integer, and let $0<C<B$ be two positive rational numbers such that $2B + {(k-1)}C = 1$.

Denote by $\mathbb{L}^0_{k,B} \subset S^2$ be the configuration of $k$ disjoint circles given by $\mathbb{L}^0_{k,B} = \bigcup_{0\leq j < k} L^{0,j}_{k,B}$, where
\begin{equation}
\label{eq-lzero}
L^{0,j}_{k,B} = (z )^{-1}(-1/2+ B+j C)\;.
\end{equation}

Let $0<a < B-C$ be a rational number. Consider the symplectic manifold $M_a = S^2 \times S^2(2a)$. Denote by $S$ the equator of $S^2(2a)$ and put $\mathbb{L}_{k,B} = \mathbb{L}^0_{k,B} \times S,$ $L^j_{k,B} = L^{0,j}_{k,B} \times S$ for $0 \leq j < k$.

For an open subset $U$ of a $2n$-dimensional symplectic manifold and a Hamiltonian $H$ supported in $[0,1] \times U$ we define the Calabi invariant as \[ \mrm{Cal}(\{\phi^t_H\}) = \int_{0}^{1} \int_U H_t \, \om^n\;. \] Moreover, $\mrm{Cal}$ defines a homomorphism $\mrm{Cal}:\til{\Ham}_c(U) \to \R.$

Finally, we call an open set $U$ {\em displaceable from a subset} $V$ if there exists
a Hamiltonian diffeomorphism $\theta$ such that $\theta(U) \cap V = \emptyset.$

Now we are ready to formulate the main result of the present section.

\blue{

\begin{thm} [Lagrangian spectral estimators]\label{thm: main spec k}

For each $k,B,a$ as above, with $B, a$ rational, there exists a map  \[c_{k,B}: \sm{[0,1] \times M_a, \R} \to \R\] satisfying the following properties:

\begin{enumerate}[label=\arabic*.]
	\item (Hofer-Lipschitz) For each $G, H \in \sm{[0,1] \times M_a, \R},$ \[|c_{k,B}(G)- c_{k,B}(H)| \leq \int_0^1  \max|G_t - H_t|\,dt.\]
	\item (Monotonicity) If $G, H \in \sm{[0,1] \times M_a, \R}$ satisfy $G \leq H$ as functions, then \[ c_{k,B}(G) \leq c_{k,B}(H).\]
	\item (Normalization) For each $H \in \sm{[0,1] \times M_a, \R}$ and $b\in \sm{[0,1],\R},$ \[c_{k,B}(H+b) = c_{k,B}(H) + \int_{0}^{1} b(t)\,dt.\]
	\item (Lagrangian control) If $(H_t)|_{L^j_{k,B}} \equiv c_j(t) \in \R$ for all $0 \leq j < k$ and $t \in [0,1],$ then \[ c_{k,B}(H) = \frac{1}{k} \sum_{0 \leq j < k} \int_{0}^{1} c_j(t)\,dt.\]
	\item (Independence of Hamiltonian) For $H \in \sm{[0,1] \times M_a, \R}$ with zero mean, the value \[c_{k,B}(H) = c_{k,B}(\phi_H)\] depends only on the class \[\phi_H = [\{ \phi^t_H \}] \in \til{\Ham}(M_a)\] in the universal cover of $\Ham(M_a)$ generated by $H.$
	\item (Subadditivity) For all $\phi,\psi \in \til{\Ham}(M_a),$ \[ c_{k,B}(\phi \psi) \leq  c_{k,B}(\phi) + c_{k,B}(\psi).\]
	
	\item (Calabi property) \label{c prop: Cal} If a Hamiltonian $H \in \sm{[0,1] \times M_a,\R}$ is supported in an open set of the form $[0,1] \times U$, where $U \subset M_a$ is disjoint from $\mathbb{L}_{k,B},$ then \[c_{k,B}(\phi_H) = - \frac{1}{\vol(M_a)} \Cal(\{\phi^t_H\}).\]
	
	\item (Controlled additivity) \label{c prop: controlled add} If $\psi = \phi_H \in \til{\Ham}(M_a)$ for a Hamiltonian $H \in \sm{[0,1] \times M_a,\R}$ supported in $[0,1] \times U $ for an open set $U \subset M_a$ disjoint from $\mathbb{L}_{k,B},$ then for all $\phi \in \til{\Ham}(M_a)$ \[ c_{k,B}(\phi \psi) = c_{k,B}(\phi) + c_{k,B}(\psi).\]

\end{enumerate}

\end{thm}

\begin{rmk}
In fact, the controlled additivity property of $c_{k,B}$ holds under the more general assumption that \[(H_t)|_{L^j_{k,B}} \equiv c_j(t) \in \R\] for all $0\leq j <k$ and $t \in [0,1].$ However, since we do not use this stronger version, we chose to omit it for simplicity of exposition.
\end{rmk}

In turn, this implies by homogenization that the following holds.
}

\begin{thm} [Lagrangian group estimators]\label{thm: main k}

For each $k,B,a$ as above, with $B, a$ rational, there exists a map  \[\mu_{k,B}: \sm{[0,1] \times M_a, \R} \to \R\] satisfying the following properties:

\begin{enumerate}[label=\arabic*.]
	\item (Hofer-Lipschitz) For each $G, H \in \sm{[0,1] \times M_a, \R},$ \[|\mu_{k,B}(G)- \mu_{k,B}(H)| \leq \int_0^1  \max|G_t - H_t|\,dt.\]
	\item (Monotonicity) If $G, H \in \sm{[0,1] \times M_a, \R}$ satisfy $G \leq H$ as functions, then \[ \mu_{k,B}(G) \leq \mu_{k,B}(H).\]
	\item (Normalization) For each $H \in \sm{[0,1] \times M_a, \R}$ and $b\in \sm{[0,1],\R},$ \[\mu_{k,B}(H+b) = \mu_{k,B}(H) + \int_{0}^{1} b(t)\,dt.\]
	\item (Lagrangian control) If $(H_t)|_{L^j_{k,B}} \equiv c_j(t) \in \R$ for all $0 \leq j < k$ and $t \in [0,1],$ then \[ \mu_{k,B}(H) = \frac{1}{k} \sum_{0 \leq j < k} \int_{0}^{1} c_j(t)\,dt.\]
	\item (Independence of Hamiltonian) For $H \in \sm{[0,1] \times M_a, \R}$ with zero mean, the value \[\mu_{k,B}(H) = \mu_{k,B}(\phi_H)\] depends only on the class \[\phi_H = [\{ \phi^t_H \}] \in \til{\Ham}(M_a)\] in the universal cover of $\Ham(M_a)$ generated by $H.$
	\item (Conjugation invariance) \label{mu prop: conj} For all $\phi,\psi \in \til{\Ham}(M_a)$ we have \[ \mu_{k,B}(\psi \phi \psi^{-1}) = \mu_{k,B}(\phi).\]
	\item (Positive homogeneity) \label{mu prop: homog} For all $\phi \in \til{\Ham}(M_a)$ and $m \in \Z_{\geq 0}$ \[ \mu_{k,B}(\phi^m) = m\cdot \mu_{k,B}(\phi).\]
	\item (Commutative subadditivity) If $\phi,\psi \in \til{\Ham}(M_a)$ commute, $\phi\psi = \psi\phi,$ then \[ \mu_{k,B}(\phi \psi) \leq  \mu_{k,B}(\phi) + \mu_{k,B}(\psi).\]
	\item (Calabi property) \label{mu prop: Cal} If a Hamiltonian $H \in \sm{[0,1] \times M_a,\R}$ is supported in an open set of the form $[0,1] \times U$, where $U \subset M_a$ is displaceable from $\mathbb{L}_{k,B},$ then \[\mu_{k,B}(\phi_H) = - \frac{1}{\vol(M_a)} \Cal(\{\phi^t_H\}).\]
\end{enumerate}

\end{thm}

\medskip\noindent
The proof is given in Section \ref{subsubsec: proof A} below.

\begin{rmk}
The rationality of $B, a$ is necessary for the Lagrangian control property and the Calabi property. In our applications to Hofer's geometry this will not lead to a loss of generality because $\Q$ is dense in $\R.$ Note also that $\psi \phi \psi^{-1}$ in the Conjugation invariance property (Property \ref{mu prop: conj}) depends only on $\phi$ and the image of $\psi$ under the natural map $\til{\Ham}(M_a) \to \Ham(M_a).$ We note that if $\psi^2 = \id$ and $\phi\psi = \psi\phi$ in $\til{\Ham}(M_a)$ then \begin{equation}\label{eq: commutative additivity torsion} \mu_{k,B}(\phi \psi) = \mu_{k,B}(\phi).\end{equation} Indeed commutative subadditivity implies that $-\mu(\psi) \leq \mu(\phi \psi) - \mu(\phi) \leq \mu(\psi),$ while positive homogeneity yields $\mu(\psi)=0.$ Moreover, note that by the Hofer-Lipschitz property one can naturally extend $\mu_{k,B}$ to a map $C^0([0,1] \times M, \R) \to \R$ satisfying a directly analogous list of properties.
\end{rmk}

\begin{rmk}
While it is not directly pertinent to our applications in this paper, it would be interesting to explicitly calculate the restriction of $c_{k,B}$ and $\mu_{k,B}$ to $\pi_1(\Ham(M_a)).$ It will be determined by the valuation of a suitable Seidel representation evaluated on Gromov's loop of infinite order in $\pi_1(\Ham(M_a))$ and respectively its homogenization.
\end{rmk}

It turns out to be useful to consider the restriction of $\mu_{k,B}$ to the space $\sm{M_a,\R}$ of Hamiltonians which do not depend on time. We recall that the Poisson bracket of two functions $F,G \in \sm{M_a,\R}$ is defined as $\{F,G\} = dF(X_G).$ The following list of properties is a direct consequence of those in Theorem \ref{thm: main k}: note that quasi-additivity and vanishing follow from commutative subadditivity and the Calabi property of $\mu_{k,B}.$


\begin{thm}[Lagrangian algebra estimators] \label{thm: main k qs}
	
	The map $\sm{M_a, \R} \to  \sm{[0,1] \times M_a, \R},$ $F \mapsto H(t,x) = F(x),$ induces a map \[\zeta_{k,B}: \sm{M_a, \R} \to \R\] \[\zeta_{k,B}(F) = \mu_{k,B}(H)\] which satisfies the following properties:
	
	\begin{enumerate}[label=\arabic*.]
		\item ($C^0$-Lipschitz) For each $G, H \in \sm{M_a, \R},$ \[|\zeta_{k,B}(G)- \zeta_{k,B}(H)| \leq  |G - H|_{C^0}.\]
		\item (Monotonicity) If $G, H \in \sm{M_a, \R}$ satisfy $G \leq H$ as functions, then \[ \zeta_{k,B}(G) \leq \zeta_{k,B}(H).\]
		\item (Normalization) \[\zeta_{k,B}(1) = 1.\]
		\item (Lagrangian control) If $H|_{L^j_{k,B}} \equiv c_j \in \R$ for all $0 \leq j < k,$ then \[ \zeta_{k,B}(H) = \frac{1}{k} \sum_{0 \leq j < k} c_j.\]
		\item (Invariance) \label{mu prop: invar} For all $\psi \in {\Ham}(M_a)$ and $H \in \sm{M_a,\R}$ we have \[ \zeta_{k,B}(H \circ \psi^{-1}) = \zeta_{k,B}(H).\]
		\item (Positive homogeneity) \label{mu prop: homog-1} For all $H \in \sm{M_a,\R}$ and $t \in \R_{\geq 0},$ \[ \zeta_{k,B}(t \cdot H) = t\cdot \zeta_{k,B}(H).\]
		\item (Quasi-additivity and vanishing) If $F,G \in \sm{M_a,\R}$ Poisson-commute, $\{F,G\} = 0,$ then \[ \zeta_{k,B}(F+G) \leq  \zeta_{k,B}(F) + \zeta_{k,B}(G)\] and if in addition $G$ is supported in $U$ displaceable from $\mathbb{L}_{k,B},$ then \[ \zeta_{k,B}(F+G) =  \zeta_{k,B}(F) + \zeta_{k,B}(G) = \zeta_{k,B}(F).\]
	\end{enumerate}
	
\end{thm}

Finally, we consider the restriction of the invariants $\mu_{k,B}, \zeta_{k,B}$ to Hamiltonians on $S^2$ by means of the stabilization by the zero Hamiltonian.

\begin{thm}\label{thm: invariants sphere}
The map $\sm{[0,1] \times S^2} \to \sm{[0,1] \times M_a},$ $F \mapsto H = F \oplus 0,$ that is, $H(t,x,y) = F(t,x)$ induces maps \[\blue{c^0_{k,B}: \sm{[0,1] \times S^2,\R} \to \R,}\]\[\mu^0_{k,B}: \sm{[0,1] \times S^2,\R} \to \R,\]\[\zeta^0_{k,B}: \sm{S^2, \R} \to \R,\] by means of \[c^0_{k,B}(F) = c_{k,B}(H),\;\;\mu^0_{k,B}(F) = \mu_{k,B}(H),\;\; \zeta^0_{k,B}(F) = \zeta_{k,B}(H).\] These maps satisfy the corresponding lists of properties as in Theorems \ref{thm: main spec k},\ref{thm: main k},\ref{thm: main k qs} with $M_a$ replaced by $S^2,$ $L^j_{k,B}$ by $L^{0,j}_{k,B},$ and $\mathbb{L}_{k,B}$ by $\mathbb{L}^{0}_{k,B}$ everywhere.

In addition, \blue{$c^0_{k,B}$ and} $\mu^0_{k,B}$ \blue{satisfy} the following stronger independence of the Hamiltonian property: \blue{$c^0_{k,B}(H)$ and} $\mu^0_{k,B}(H)$ for a mean-normalized Hamiltonian $H \in \sm{[0,1] \times S^2, \R}$ depend only on $\phi = \phi^1_H \in \Ham(S^2).$

As a function \[c^0_{k,B}: \Ham(S^2) \to \R\] it satisfies the subadditivity, Calabi, and controlled additivity properties. As a function \[\mu^0_{k,B}: \Ham(S^2) \to \R\] it satisfies the conjugation invariance, positive homogeneity, commutative subadditivity, and the Calabi properties. In particular, \[|\zeta^0_{k,B}(H)| \leq d_{\mrm{Hofer}}(\phi^1_H,\id)\] for all $H \in \sm{S^2,\R}.$
\end{thm}

\begin{proof}
The proof of all statements is immediate, except for stronger independence of the Hamiltonian. To this end, we observe that by a classical result of Smale $\pi_1(\Ham(S^2)) \cong \zt.$ Let $\psi$ be its generator. Since $\pi_1(\Ham(S^2))$ lies in the center of $\til{\Ham}(S^2),$ the result for $\mu^0_{k,B}$ follows from \eqref{eq: commutative additivity torsion}. For $c^0_{k,B}$ we have $c^0_{k,B}(\psi) = 0$ by Lagrangian control, since $\psi$ is generated by the mean-zero Hamiltonian $F = z: S^2 \to [-1/2,1/2].$ Hence by subadditivity, for all $\phi \in \til{\Ham}(S^2)$ we have $c^0_{k,B}(\phi\psi) \leq c^0_{k,B}(\phi).$ Replacing $\phi$ by $\phi\psi$ and using $\psi^2 = \id,$ we obtain the inequality $c^0_{k,B}(\phi) \leq c^0_{k,B}(\phi \psi)$ in the reverse direction.
\end{proof}

\color{black}
%

\begin{rmk}\label{rem-qs}
We observe that the maps $\zeta^0_{k,B},$ and hence also $\zeta_{k,B},$ are not partial symplectic quasi-states for $k>1.$ Indeed, $\zeta^0_{k,B}$ equals $1/k$ for the cut-off
of the indicator function of a small neighbourhood of $L^{0,0}_{k,B}$, a circle
of our configuration having the smallest area. But this circle is displaceable,
a contradiction with the vanishing axiom for quasi-states.
However, it is not hard to see that the choices involved in the Floer data defining $\zeta^0_{1,1/2}$ can be chosen in such a way that it coincides with the symplectic quasi-state $\zeta_0$ on $S^2$ (which is in fact unique: see \cite[Exercise 5.4.29]{PolterovichRosen}).
\end{rmk}

\section{Hofer's geometry: proofs and further results}\label{sec-H}

Here we apply the techniques of Lagrangian estimators to the proof of our main
applications to Hofer's geometry. Note that for proofs of Theorems \ref{thm: main-L}
and \ref{thm: main2-L} we need the simplest Lagrangian configurations consisting of
two circles.

\subsection{Proof of Theorem \ref{thm: main-L}}\label{subsec-thm-main-L}

\medskip
\noindent {\sc Step 1: Construction.}
We start with a more explicit formulation of the theorem.
Consider a symmetric interval $I = (-b, b)$ for $b < 1/6.$ Let \[\mathcal{F}_I \subset C^{\infty}_c(I)\] be the space of {\em even} compactly supported smooth functions on $I.$ In other words $\mathcal{F}_I$ consists of functions $h \in C^{\infty}_c(I)$ satisfying $h(x) = h(-x)$ for all $x \in I.$ Endow $\cF_I$ with the $C^0$ norm \[ |h|_{C^0} = \max_{I} |h|,\] which induces the distance function $d_{C^0}(h_1,h_2) = |h_1- h_2|_{C^0}.$  Observe that the group $\cl{G} =  C^{\infty}_c((0,b))$ with the $C^0$ distance naturally embeds into $\cl{F}_I$ by even extension.

Consider the standard symplectic sphere $(S^2,\om)$ of total area $1.$ It admits a Hamiltonian $S^1$-action whose zero-mean normalized moment map $z:S^2 \to \R$ has image $[-1/2,1/2].$ We define the following embedding \[\Phi:\cl{F}_I \to \Ham(S^2),\] keeping in mind that the Hofer metric works with zero-mean Hamiltonians. We first build an embedding of $\cF_I$ to the space $\cI$ of even functions of integral zero in $C^{\infty}([-1/2,1/2]).$ To a function $h \in \cF_I$ we associate $h^{\#} \in \cI$ defined by $h^{\#}|_{(-b,b)} = h,$ $h^{\#}|_{(1/2-b,1/2]}(x) = - h(1/2-x),$ $h^{\#}|_{[-1/2,-1/2+b)}(x) = - h(-1/2-x),$ extended by zero to $[-1/2,1/2].$ Now, for $h \in \cF_I$ we consider the mean-zero Hamiltonian $\Gamma(h) \in C^{\infty}(S^2,\R)$ given by \[ \Gamma(h) = h^{\#} \circ z,\] and let \[\Phi(h) = \phi^1_{\Gamma(h)}\] be the time-one map of $\Gamma(h).$ It is immediate by construction that this map $\Phi: \cF_I \to \Ham(S^2)$ is a homomorphism and for all $h_1, h_2 \in \cl F$ \begin{equation}\label{eq: obvious upper bound}d_{\rm{Hofer}}(\Phi(h_1),\Phi(h_2)) = d_{\rm{Hofer}}(\Phi(h_1-h_2),\id) \leq |h_1 - h_2|_{C^0}.\end{equation}
We {\it claim} that the monomorphism of groups \[\Phi: (\cF_I, d_{C^0}) \hookrightarrow (\Ham(S^2),d_{\rm{Hofer}})\] is an isometric embedding.

\medskip\noindent
{\sc Step 2: Proof of the claim.}
Note that by \eqref{eq: obvious upper bound}, the main step in the proof of the
claim is the inequality
\begin{equation}\label{eq: main} d_{\rm{Hofer}}(\Phi(h),\id) \geq |h|_{C^0}\;\; \forall h \in \cF_I \;.\end{equation}
As $|(-h)|_{C^0} = |h|_{C^0}$ and \[d_{\rm{Hofer}}(\Phi(-h),\id) = d_{\rm{Hofer}}(\Phi(h)^{-1},\id) = d_{\rm{Hofer}}(\Phi(h),\id)\] for all $h \in \cF_I,$ it is sufficient to prove \eqref{eq: main} under the assumption that $|h|_{C^0} = h(x_0) > 0.$ Note that in this case $h(x_0)= h(-x_0),$ and hence either $x_0 = 0,$ or we can assume that $x_0 \in (0,b).$ Consider $B = 1/2 - x_0.$ For\footnote{When $x_0 = 0,$ \eqref{eq: main} follows from a result of the first author \cite{Polterovich-S2}. See \cite{EntovPolterovichCalabiQM,EntovPolterovich-rigid} and \cite{LeclercqZapolsky} for alternative proofs.} $x_0 \in [0,b)$ we consider $\zeta_i = \zeta^0_{2,B_i},$ where $B_i = 1/2 - x_i,$ and $x_i \in (0,b)$ is a sequence of rational numbers converging to $x_0.$ Here we fix a rational parameter $0<a<1/2-3b$ for defining $\zeta^0_{2,B_i}$ as in Theorem \ref{thm: invariants sphere}. By continuity of $h$ we now have \[h(x_i) \xrightarrow{i \to \infty} h(x_0) = |h|_{C^0}.\] Now by the Lagrangian control property of $\zeta_{i}$ we have \[\zeta_{i}(\Gamma(h)) = h(x_i),\] and by the Hofer-Lipschitz and independence of Hamiltonian properties \[d_{\mrm{Hofer}}(\Phi(h)) \geq \zeta_{i}(\Gamma(h))\] for all $i.$ Therefore taking limits as $i \to \infty$ we obtain \[d_{\mrm{Hofer}}(\Phi(h)) \geq |h|_{C^0}\] as required. This finishes the proof. \qed

\medskip

\begin{rmk} We note that the interval $I = (-1/6,1/6)$ in the claim of Step 1
is the best possible using the method we describe in this paper. Indeed, the disconnected Lagrangian given by $z^{-1}( \{\pm (1/6+\delta) \})$ yields a Lagrangian in the symmetric square that is displaceable from itself. (See Section \ref{sec: prelim} for a description of the framework.). Moreover, Sikorav's trick in Section \ref{sec-lp} below shows that this embedding is not isometric for $b>1/6.$
\end{rmk}

\begin{rmk}\label{rmk: ball}  Note also that for any odd function $h(z) \in \sm{S^2,\R},$ that is $h(-z) = - h(z),$ $\phi^1_h$ is conjugate to its inverse. Indeed for the involution $R \in \Ham(S^2)$ given by $R(x,y,z) = (x,-y,-z)$ we have $h\circ R^{-1} = - h$ so $(\phi^1_{h})^{-1} = \phi^1_{-h} = \phi^1_{h \circ R^{-1}} = R \circ \phi^1_h \circ R^{-1}.$ Therefore, $d_{\mrm{Hofer}}(\phi^1_{2h},\id) \leq C = 2 d_{\mrm{Hofer}}(R,\id).$ Hence all such odd autonomous Hamiltonians generate one-parametric subgroups in the ball of Hofer of radius $C$ around the identity. From this perspective, it is natural that our construction is based on even functions.  \end{rmk}

\subsection{Proof of Theorem \ref{thm: main2-L}}
Let $I = (-b, b)$ for $b<1/6$ and $0<a<1/2-3b.$
Arguing as in the proof of Theorem \ref{thm: main-L} we get that
the monomorphism $\til{\Psi}:\cl{F}_I \to \til{\Ham}(M_a)$ covering $\Psi$ is an isometric embedding.

In order to pass from the universal cover to the group itself, we shall
restrict ourselves to the subspace $\cF^0_I \subset \cF_I$ consisting of functions
$h$ with $h(0)=0$. Note that the image of $\cl{G} =C^{\infty}_c \left( (0,b)\right)$ in  $\cl{F}_I $
lies in $\cF^0_I$.

Let us show that the monomorphism of groups \[\Psi = \iota \circ \Phi: (\cF^0_I, d_{C^0}) \hookrightarrow (\Ham(M_a),d_{\rm{Hofer}})\] is a bi-Lipschitz embedding.
It suffices to show that there exists $K \geq 1$ such that for all $h \in \cl{F}_I$
\begin{equation} \label{eq-toproof}
d_{\mrm{Hofer}}(\id, \Psi(h)) \geq \frac{1}{K} |h|_{C^0}\;.
\end{equation}

Suppose without loss of generality that
$||h||_{C^0} = h(r)>0$ with $r \in (0,b)$. 

Consider the invariant $\mu_{2,B}: \til{\Ham}(M_a) \to \R$ provided by Theorem \ref{thm: main k} for $$1/2 > B = 1/2-r \geq 1/2 - b\;.$$

In order to proceed further, we have to understand the effect of the fundamental group
$\pi_1(\Ham(M))$. It is a finitely generated abelian group. In fact by \cite[Theorem 1.1]{AbreuMcDuff} we have $\pi_1(\Ham(M)) \cong \Z \oplus \Z/2\Z \oplus \Z/2\Z.$ The $\Z/2\Z$ terms appear from the natural maps $\pi_1(\Ham(S^2)) \to \pi_1(\Ham(M))$ corresponding to the two components of $M = S^2 \times S^2.$ The $\Z$ term is the well-known Gromov loop \cite{GromovPseudohol}, investigated in detail by Abreu and McDuff \cite{McDuff-examples, AbreuMcDuff}. Set $\cl T = \zt \oplus \zt$ for the torsion part of $\cl G = \pi_1(\Ham(M)),$ and let $\cl A = \cl G/ \cl T \cong \Z$ be its free part. Note that $\pi_1(\Ham(M)) \subset Z(\til{\Ham}(M))$ is a central subgroup. As in the proof of Theorem \ref{thm: main-L} it is easy to see that $\mu_{2,B}$ vanishes on $\cl T$ and by commutative subadditivity descends to $\til{\Ham}(M)/ \cl T.$ However, the same is {\em not} clear for $\cl A.$ We proceed differently.

Observe that by \cite[Theorem 1.2]{OstroverAGT} there exists a homogeneous Calabi quasi-morphism $\rho: \til{\Ham}(M) \to \R$ that is $1$-Lipschitz in the Hofer metric and restricts to a non-trivial homomorphism $\cl G \to \R.$ It vanishes on the torsion $\cl T$ so we can consider it to be a homomorphism $\rho: \cl A \to \R$. Choosing a generator $g$ of $\cl A,$ we have $\rho(g^k) = k \cdot \rho(g)$ for a positive constant $\rho(g).$ It is also
known that $\rho(\Psi(h)) = h(0)=0$ for $h \in \cF^0_I$. {This follows from $\rho$ yielding a symplectic quasi-state on $\sm{M_a,\R}$ and $S \times S \subset M_a,$ where $S \subset S^2$ is the equator, being a stem (see \cite{EntovPolterovich-rigid,EntovPolterovich-intersections,PolterovichRosen}).}

Define the map \[\nu_{r}: \til{\Ham}(M) \to \R^2,\] \[\nu_{r}(\til{\phi}) =  (\mu_{2,B}(\til{\phi}), \rho(\til{\phi})).\] Let $\til{\Psi}:\cl{F}_I^0 \to \til{\Ham}(M)$ be the homomorphism covering $\Psi.$

Observe that $\nu_{r}$ is $1$-Lipschitz in Hofer's metric, where $\R^2$ is endowed with the $l_{\infty}$ norm.

By the above-mentioned properties of $\rho$, and by Lagrangian control of $\mu_{2,B}$, we can calculate for an arbitrary element $\til{\Psi}(h) g^k f$ covering $\Psi(h),$ where $f \in \cl{T},$ that
\[\nu_{r}(\til{\Psi}(h) g^k f) = \nu_{r}(\til{\Psi}(h) g^k) = (\mu_{2,B}(\til{\Psi}(h)g^k), \rho(\til{\Psi}(h)g^k)) = (\mu_{2,B}(\til{\Psi}(h)g^k), k \rho(g)).\]

This implies that \[ |\nu_{r}(\til{\Psi}(h) g^k f)| \geq \max \{ |h|_{C^0} - |k| d_{\mrm{Hofer}}(g,\id), |k| \rho(g)\} \geq  C_2 |h|_{C^0},\] for $C_2 = \frac{\rho(g)}{\rho(g) + d_{\mrm{Hofer}}(g,\id)},$ the last step being an easy optimization in $|k|.$

Therefore we have \[C_2 |h|_{C^0} \leq |\nu_{r}(\til{\Psi}(h)g^k f)| \leq d_{\mrm{Hofer}}(\til{\Psi}(h)g^k f,\id),\] and hence \[d_{\mrm{Hofer}}({\Psi}(h),\id) \geq \frac{1}{K} |h|_{C^0}\] for $K = 1/C_2 \geq 1.$ This finishes the proof.
\qed

\subsection{Almost flats in the kernel of Calabi}
Fix the interval \[I_k = \left(-1/2+1/{(k+1)},-1/2+1/{k} \right).\] Consider $\mu_{k,B}$ with $1/k > B > 1/{(k+1)}.$ Note that in this case $B>C,$ so $\mu_{k,B}$ is well-defined. We prove that the space of all functions $\cl G_k \subset C^{\infty}_c(I_k)$ with zero mean admits an isometric embedding into $(\Ham(S^2), d_{\mrm{Hofer}}).$ Recall that for a proper open subset $U \subset S^2$ its symplectic form is exact, and the Calabi homomorphism \[\mrm{Cal}_{U}: \Ham_c(U) \to \R\] is defined as $\mrm{Cal}_{U}(\phi) = \mrm{Cal}_{U}(\{\phi^t_H\})$ for any $H \in \smc{[0,1] \times U, \R}$ with $\phi = \phi^1_H.$ Observe that by the natural constant extension we have the inclusion $\Ham_c(U) \to \Ham(S^2).$ In particular, $\ker(\mrm{Cal}_U)$ can and shall be considered to be a subgroup of $\Ham(S^2).$ Starting from the following result we require configurations $\mathbb{L}_{k,B}$ for all values of $k.$ 

We call a map $f:(X,d_X) \to (Y,d_Y)$ of metric spaces {\it almost isometric} if there exists a constant $D \geq 0$ such that \[d_X(p,q)-D \leq d_Y(f(p),f(q)) \leq d_X(p,q)+D\] for all $p,q \in X.$

\begin{thm}\label{thm: main ker Cal}
Let $\Phi_k: (\cl G_k, d_{C^0}) \to \left(\Ham(S^2), d_{\mrm{Hofer}}\right)$ be given by \[ \Phi_k(h) = \phi^1_{H}\] for $H = k \cdot h\circ z.$ Then $\Phi_k$ is a bi-Lipschitz almost isometric group embedding, whose image lies in $\ker(\mrm{Cal_{D}}),$ where $D = D_{1/k}$ is the open cap of area $1/k$ around the south pole. Furthermore, for each proper open set $U \subset S^2$ there is a bi-Lipschitz almost isometric embedding of $\smc{I}$ of an open interval $I$ into $\ker(\mrm{Cal_{U}}).$
\end{thm}

A similar result is proved in \cite{CGHS2}.

The proof of this statement is similar to that of Theorem \ref{thm: main-L}, with the additional observation that by the conjugation invariance of $\mu_{k,B}$ we may suppose that the open set $U$ contains $D_{1/k}$ for $k$ sufficiently large. Moreover, passing from smooth functions on an interval, say $ \smc{(-\frac{1}{2}+\frac{1}{k+1}, -\frac{1}{2} + \frac{1}{2(k+1)}+\frac{1}{2k})}$ to functions in $\cl{G}_k$ can be easily carried out by odd extension about the midpoint ${\upsilon_k = -\frac{1}{2} + \frac{1}{2(k+1)}+\frac{1}{2k}}$ of $I_k.$ On one hand by use of Lagrangian estimators for the lower bound, and the evident upper bound on the Hofer norm via the Hamiltonian, we obtain \[ |h_1-h_2|_{C^0} \leq d_{\mrm{Hofer}}(\Phi_k(h_1),\Phi_k(h_2)) \leq k |h_1-h_2|_{C^0}\] for all $h_1, h_2 \in \cl G_k.$ Hence $\Phi_k$ is bi-Lipschitz. On the other hand, estimating $d_{\mrm{Hofer}}(\Phi_k(h_1),\Phi_k(h_2)) = d_{\mrm{Hofer}}(\Phi_k(h_1-h_2),\id)$ via Sikorav's trick (see Section \ref{sec-lp}) we obtain 
\[ d_{\mrm{Hofer}}(\Phi_k(h_1),\Phi_k(h_2)) \leq |h_1-h_2|_{C^0} + D_k\] for a constant $D_k$ depending only on $k.$ This proves that $\Phi_k$ is almost isometric. (We remark that one can show that $D_k < 2$ for all $k.$)

Finally, analogously to the proof of Theorem \ref{thm: main2-L}, we can show that these almost isometric embeddings extend to stabilizations on the level of universal covers and remain bi-Lipschitz embeddings on the level of groups.

\begin{thm}\label{thm:main2 k} Let $k \geq 2,$ $1/{(k+1)} < B < 1/k$ and $0<a<\frac{k+1}{k-1}B - \frac{1}{k-1}.$ The monomorphism of groups \[\Psi_k = \iota \circ \Phi_k: (\cl G_k, d_{C^0}) \hookrightarrow (\Ham(M_a),d_{\rm{Hofer}})\] is a bi-Lipschitz embedding. In fact, there exists $K \geq 1$ such that for all $h_1,h_2 \in \cl{G}_k$ \[ \frac{1}{K} |h_1-h_2|_{C^0} \leq d_{\mrm{Hofer}}(\Psi_k(h_1),\Psi_k(h_2)) \leq k |h_1-h_2|_{C^0}. \]
	At the same time, the monomorphism $\til{\Psi}_k:\cl{G}_k \to \til{\Ham}(M_a)$ covering $\Psi_k$ is an almost isometric embedding.
	
\end{thm}

\begin{qtn}\label{qtn: FOOO}
In \cite{FO3-degeneration} a family of non-displaceable Lagrangian tori was constructed on $S^2 \times S^2.$ It is possible to prove that these tori yield the existence of large flats in $\Ham(S^2 \times S^2).$ These flats do not come by stabilization from $S^2.$ Is it possible to prove that they cannot be at a finite Hofer distance from a flat supported in an arbitrarily small neighborhood of a symplectic divisor of the form $\{pt\} \times S^2$?
\end{qtn}

\subsection{The asymptotic Hofer norm}\label{subsec-asymptotic}

Let $(M,\om)$ be a closed symplectic manifold. The asymptotic Hofer (pseudo-)norm $\nu(\til{\phi})$ of an element $\til{\phi} \in \til{\Ham}(M,\om)$ is defined as follows: \[ \nu(\tilde{\phi}) = \lim_{m\to \infty} \frac{1}{m} \til{d}_{\mrm{Hofer}}(\til{\phi}^m,\id),\] where $\til{d}_{\mrm{Hofer}}$ is the Hofer pseudo-norm on $\til{\Ham}(M,\om).$ For $\phi \in \Ham(M,\om)$ we define its asymptotic Hofer norm $\nu(\phi)$ similarly, or alternatively $\nu(\phi) = \inf \nu(\til{\phi})$ where the infimum runs over all $\til{\phi} \in \til{\Ham}(M,\om)$ covering $\phi.$

The Hofer norm is known to yield quantitative invariants of subsets in symplectic manifolds, for instance the displacement energy (\cite{HoferMetric}, see also \cite{P-book}). It turns out that the asymptotic Hofer norm produces an invariant of subsets of symplectic manifolds controlling their packing number. 

Let $A \subset M$ be a compact subset. Its packing number $k(A) \in \bb N \cup \{\infty\}$ is defined as the maximal $k$ such that there exist $\theta_1,\ldots,\theta_k \in \Ham(M,\om)$ with $\theta_1(A),\ldots,\theta_k(A)$ pairwise disjoint ($\theta_i(A) \cap \theta_j(A) = \emptyset$ for all $i \neq j$). 

Let $C^{\infty}_0(M;\mathbb{R})$ denote the space of mean-zero smooth functions on $M.$ For $f \in C^{\infty}_0(M;\mathbb{R}),$ set $\til{\phi}_f \in \widetilde{\mathrm{Ham}}(M,\omega)$ for the class of the time-one Hamiltonian isotopy which it generates. 

\begin{df}\label{def: alpha invariant}
Let $A \subset M$ be a proper compact set. Its $\alpha$-invariant is \[ \alpha(A) = \inf_{U \supset A} \sup\{ {\nu}(\til{\phi}_f)\;|\; f \in C^{\infty}_0(M;\mathbb{R}),\;|f|_{C^0} = 1,\; \mathrm{supp}(f) \subset U,\; f|_A \equiv 1 \},\] where the infimum runs over all open neighborhoods $U$ of $A.$
\end{df}

\begin{rmk}
A more elaborate alternative invariant $\alpha'(A)$ would be the infimum over open neighborhoods $U$ of $A$ of constants $a>0$ such that the map ${\iota_U: \ker(\Cal_U) \to \til{\Ham}(M,\om)}$ is coarse Lipschitz with constant $a>0,$ that is \[d_{\mrm{Hofer}}(\iota_U(\til{\phi}),\iota_U(\til{\psi})) \leq a \cdot d_{\mrm{Hofer}}(\til{\phi},\til{\psi}) + b\] for all $\til{\phi},\til{\psi} \in \ker(\Cal_U)$ and $b \geq 0$ independent of $\til{\phi},\til{\psi}.$ It is easy to see that $\alpha(A) \leq \alpha'(A)$ and Proposition \ref{prop:k and alpha} below applies to $\alpha'(A)$ as well. 
\end{rmk}


In view of {\em Sikorav's trick} \cite{Sikorav,HZ-book}, which we recall below in detail for the reader's convenience, we obtain the following. 

\begin{prop}\label{prop:k and alpha}
The packing number of $A$ satisfies \[k(A) \leq \Big\lfloor \frac{1}{\alpha(A)}\Big\rfloor,\] where the right hand side is by convention $\infty$ if $\alpha(A) = 0.$	
\end{prop}

\begin{proof}[Proof of Proposition \ref{prop:k and alpha}]

It is sufficient to prove that $k(A) \leq \frac{1}{\alpha(A)}.$ If $\alpha(A) = 0,$ there is nothing to prove. Hence we suppose that $\alpha(A) > 0.$ Then we can prove the equivalent statement that $\alpha(A) \leq 1/k(A).$

Suppose that $M$ can be packed by $l \leq k$ copies of $A.$ Namely, let $\theta_1 = \id, \ldots, \theta_l \in \Ham(M,\om)$ yield an $l$-packing of $A$ in $M:$ $\{ \theta_j(A) \}_{1 \leq j \leq l}$ are pairwise disjoint. Let $U \supset A$ be a sufficiently small open neighborhood of $A$ so that $\{ \theta_j(U) \}_{1 \leq j \leq l}$ are pairwise disjoint. Set $U_j = \theta_j(U)$ for $1\leq j\leq l.$ Let $f \in C^{\infty}_0(M;\mathbb{R}),\;|f|_{C^0} = 1,\; \mathrm{supp}(f) \subset U,\; f|_A \equiv 1$ be as in the definition of $\alpha(A).$ For $1\leq j\leq l,$ $f_j = f \circ \theta_j^{-1}$  satisfies $\mrm{supp}(f_j) \subset U_j.$ 

Let us show that $\nu(\til{\phi}_f,\id) \leq 1/l,$ as calculated in $\til{\Ham}(M,\om),$ which implies our claim. Indeed by the conjugation invariance of Hofer's pseudo-metric on $\til{\Ham}(M,\om),$ \[ \dh{\phi_{f}^{lt}} {\phi^t_{f_1}\cdot...\cdot \phi^t_{f_l}} \leq C,\] where $C \leq 2\left(\dh{\theta_1}{\id} + \ldots + \dh{\theta_l}{\id}\right)$ is {\em independent of $t.$} However \[ \phi^t_{f_1}\cdot...\cdot\phi^t_{f_l} = \phi^t_{F} ,\] \[ F= f_1 + \ldots +f_l.\] As $|F|_{C^0} = 1$ we obtain that $\nu(\phi_f) = \frac{1}{l} \nu(\phi_F) \leq 1/l.$ \end{proof}


Next, we study the asymptotic Hofer norm in the following situation. Let $M = S^2$ with a round area form of total area $1.$ We write $z$ for the vertical coordinate scaled by a factor of $1/2.$ Let us fix $B> C > 0$ with $2B+ (k-1)C=1$.
Consider a Lagrangian configuration $\sqcup_{1 \leq j \leq k} \{z = z_j \},$ where
\begin{equation}\label{eq-zets}
z_j = -1/2 + B+(j-1)C, \;\; j=1, \dots ,k\;.
\end{equation}
Denote by $\sigma_{B,C}$ the measure $\frac{1}{k}\sum_{j=1}^k \delta_{z_j}$.
By using Lagrangian estimators $\mu_{B_i,k}$ as in Theorem \ref{thm: main k}, for rational $B_i \xrightarrow{i \to \infty} B,$
we get that for every smooth function $h = h(z)$ on $S^2$ with zero mean
the asymptotic Hofer norm satisfies
\begin{equation}
\label{eqref-asH}
\nu(\phi_h)= \lim_{t\to+\infty} \frac{{d}_{\rm Hofer}(\phi_h^t,\id)}{t} \geq \int h d\sigma_{B,C}\;.
\end{equation}
Let us emphasize that in this definition the flow $\phi_h^t$ naturally lifts to the universal
cover $\til{\Ham}$ and we consider Hofer's metric there.

\begin{qtn}\label{qtn-asymp}  {\rm Is this estimate is sharp?}
\end{qtn}

As a test, we fix small $\delta >0$ and
$$r \in \left(\frac{1}{k+1},  \frac{1}{k}\right)\;,$$
and consider $h_{r,\delta}$ to be a smoothing of the indicator function of $$[-1/2+r-\delta,-1/2+r+\delta]\;$$
which we arrange to have zero mean by extending it to $[-1/2+r-\delta,-1/2+r+5\delta]$ as an odd function about the midpoint $-1/2+r+2\delta$ of the interval.
Choose $B= r$, put $C = (1-2B)/(k-1)$. Note that $h_{r,\delta}$ has $|h_{r,\delta}|_{C^0} = 1$ and equals $1$ on the circle $K_r:= L^{0,0}_{k,B}$ of area $r$ and is supported in its
$\delta$-neighbourhood. Since $B>C$, we can apply inequality
\eqref{eqref-asH} with the measure $\sigma_{B,C}$ and get
$$u(r) := \liminf_{\delta\to 0}\nu(\phi_{h_{r,\delta}}) \geq \frac{1}{k}\;.$$
At the same time, note that there exists a packing of $S^2$ by $k$ copies of the support of $h$,
so by Proposition \ref{prop:k and alpha}
$u(r) \leq 1/k$. Thus $u(r)=1/k$, so the estimate is sharp.

Presumably, progress in the direction outlined in Section \ref{subsec-other}
will yield efficient lower bounds on the asymptotic Hofer norm for more general autonomous
Hamiltonians.

\subsection{Stabilization}
Consider now the {\it stabilization} of the Hamiltonian diffeomorphism $\phi_{h_{r,\delta}}$
constructed in Section \ref{subsec-asymptotic}. Fix a closed symplectic manifold
$(P,\Omega)$ and put
$$\Phi_{r,\delta}:= \phi_{h_{r,\delta}} \times \id \in \Ham(S^2 \times P)\;.$$
Put
$$u_P(r) = \lim_{\delta\to 0}\nu(\Phi_{r,\delta})\;.$$

\begin{qtn} \label{qtn-stab}
	{\rm Is $u_P(r) >0$?}
\end{qtn}

If the answer is affirmative, it would be interesting to calculate or estimate this quantity.
Our method, based on spectral invariants in Lagrangian Floer theory on orbifolds, yields ``yes" when $(P,\Omega)$ is a $2$-sphere of
area $2a$ with $B+a > C$. In fact, we have in this case $u_P(r) = 1/k$ as above.

Now we discuss another version of the stabilization. Let $S \subset P$ be a closed Lagrangian submanifold, and let $f_W$ be a smoothing of the characteristic function of a small Weinstein neighbourhood $W$ of $S$. Consider the Hamiltonian $h_{r,\delta}f_W$. We denote by $u_S(r)$ the lower limit of the corresponding asymptotic Hofer norm when $r \to 0$ and $W$ shrinks to $S$.

\begin{qtn} \label{qtn-stab-1}
	Under which assumptions on $S$, one has $u_S(r) >0$?
\end{qtn}

For instance, when $P$ is $S^2$ of area $2a$ with $B> C+a$ and $S$ is the equator,
Hamiltonians $h_{r,\delta}f_W,$ are concentrated near $\Lambda_r:=K_r \times S$,
and the same argument based on Theorem \ref{thm: main k} and Proposition \ref{prop:k and alpha} yields
\begin{equation}\label{eq-K}
u_S(r) = 1/k\;.
\end{equation}

\subsection{Lagrangian packing via Hofer's geometry}\label{sec-lp}

Here we use identity \eqref{eq-K} for the asymptotic Hofer norm in order to
deduce Theorem \ref{cor-lp}.

\medskip\noindent
{\it Proof of Theorem \ref{cor-lp}:} Equation \eqref{eq-K} implies that $\alpha(\Lambda_r) \geq 1/k,$ whence by Proposition \ref{prop:k and alpha}, $k(\Lambda_r) \leq k.$ \qed


%

\begin{rmk}
Theorem \ref{cor-lp} can be interpreted as follows\footnote{We thank Ivan Smith for bringing this interpretation to our attention.}. Consider $m=k+1$ Hamiltonian copies $L_1 = \phi_1 (\Lambda_r),\ldots, L_m=\phi_m (\Lambda_r),$ where $\phi_j \in \Ham(M_a)$ for all $1 \leq j \leq m,$ of $\Lambda_r$ in $M_a.$ Then the collection $\{L_j\}_{1 \leq j \leq m}$ has the following ``Borromean" property: for every proper subset $I$ of $\{1,\ldots, m\}$ there exist $\{\theta_j \in \Ham(M_a) \}_{j \in I}$ such that $\{ \theta_j(L_j) \}_{j \in I}$ are pairwise disjoint, but for $I = \{1,\ldots,m\}$ such a disjoinment does not exist. Of course this is part of a general phenomenon, which arises for any compact set $A \subset M$ with finite packing number $k(A),$ if we take $m=k(A)+1.$
\end{rmk}

\begin{qtn}
Is it possible to prove that for $a$ larger than $B-C$ one can pack $k+1$ Hamiltonian images of $\Lambda$ or more? The recent methods of Hind and Kerman \cite{Hind-Kerman-wp} might help produce packings of this kind.
\end{qtn}

\section{Lagrangian Poincar\'{e} recurrence: proof}\label{sec-lr}

{\it Proof of Theorem \ref{cor-lr}:} Denote by $K$ the complete graph with vertices $\mathbb{Z}_{\geq 0}$. Edge-color $K$ as follows:
	the edge $ij$ is blue if $\phi^i \Lambda \cap \phi^j \Lambda \neq \emptyset$, and it is red otherwise. By Corollary \ref{cor-lp}, this coloring does not possess any red complete subgraph with $k+1$ vertices. Fix a maximal red complete subgraph, say, $Q$.
	
	Since $Q$ is maximal, each vertex outside $Q$ is connected to some vertex in $Q$ by a blue edge.
	Put $N = mk$, and consider the graph $B_N$ with vertices $\{0,\dots,N\}$ connected only by the blue edges.
	The positive integer $m$ will play the role of the large parameter in the proof. In particular,
	we assume that the maximal element $q$ of $Q$ is $\leq mk$.
	
	The number of vertices in $B_N \setminus Q$ is at least $(m-1)k + 1$. Denote by $d$ the maximal
	degree of a vertex from $Q$ in $B_N$. Then, by counting outcoming blue edges from $Q$ we get
	$(m-1)k+1 \leq dk$ which yields $d \geq m$. It follows that some vertex $p \in Q$ has at least
	$m$ blue outcoming edges in $B_N$.
	
	Note now that the coloring is invariant under positive translations. It follows that $0$ has at least $m-q$ outcoming blue edges (we can lose at most $q$ edges as roughly speaking the corresponding  vertices will become negative after the shift by $-p$),   yielding
	$$|\cR_\phi \cap [0,mk]| \geq m-q\;.$$
	We conclude the proof by noticing that $(m-q)/mk \to 1/k$ as $m \to +\infty$.
\qed

\medskip

Furthermore, we observe that the same estimates work in general whenever a subset $\Lambda \subset X$ of a set $X$ cannot be $(k+1)$-packed into $X$ by powers of a given invertible map $\phi: X \to X.$ For instance this is the case when $(X,\nu)$ is a measure space of total measure $1,$ $\Lambda$ is a subset of positive measure, $k = \lfloor 1/\nu(\Lambda) \rfloor,$ and $\phi$ is any invertible measure-preserving transformation. In this setting a stronger statement follows from the Ergodic Theorem \cite[Theorem 1.2]{Bergelson}: there exists a sequence $i_m$ of density $\geq 1/\nu(\Lambda)$ such that $\nu(\Lambda \cap \phi^{-i_1} \Lambda \cap \ldots \cap \phi^{-i_m} \Lambda) > 0$ for all $m.$ However, in our Lagrangian situation the same method is not applicable: indeed our Lagrangian $\Lambda$ is a measure-zero subset which does not bound a positive-measure subset. For a different example in symplectic topology, we could consider $X$ to be a symplectic manifold, $\phi \in \Symp(X)$ a symplectomorphism, $\Lambda$ an open ball of a given capacity, and $k$ its packing number. That this sometimes gives sharper bounds than simply the volume constraint is one of the paradigms of modern symplectic topology, initiated in \cite{GromovPseudohol,McDuffPolterovichPacking,Biran-ECM}. 

\section{$C^0$-continuity and non-simplicity}\label{sec: C0}

We observe that the collection of Lagrangian spectral estimators $c^0_{k,B}$ contains sufficient data to provide new $C^0$-continuous invariants on $\Ham(S^2)$ that extend to the group of area-preserving homeomorphisms $G_{S^2}$ of the two-sphere. Following the strategy of  \cite{CGHS1,CGHS2} this is shown to yield alternative proofs of the non-simplicity of the group of compactly supported area-preserving homeomorphisms of the two-disk, known as the ``simplicity conjecture" \cite[Problem 42]{McDuffSalamonIntro3} by way of proving the ``infinite twist conjecture" \cite[Problem 43]{McDuffSalamonIntro3}, as well as that of $G_{S^2}.$ We refer to the above references for the original proofs of these conjectures by using periodic Floer homology, as well as for ample further information about the non-simplicity questions and their historical context. We shall also rely on the following result, \cite[Lemma 3.11]{CGHS2}, proven by soft fragmentation methods. 

\begin{lma}\label{lma: CGHS2}
Let $U \subset S^2$ be a disk. Then for all $\eps > 0$ there exists $\delta > 0$ such that if $\phi \in \Ham(S^2)$ satisfies $d_{C^0}(\phi,\id) < \delta$ then there exists $\psi$ supported in $U$ such that $d_{\mrm{Hofer}}(\phi,\psi) < \eps.$
\end{lma}

This allows us to prove the following $C^0$-continuity result. Consider the $C^0$-closure groups $G_{S^2} = \Ham(S^2)^{C^0}$ and $G_{\D^2} = \Ham_c(\bb D^2)^{C^0}$ inside the corresponding homeomorphism groups. It is a well-known fact that these groups coincide with the groups of orientation and area-preserving homeomorphisms of $S^2$ and of area-preserving homeomorphisms of $\bb D^2$ with compact support (see e.g. \cite{MullerApprox} and references therein). 

\begin{thm}\label{thm-extends}
The map \[\tau_{k,k',B,B'}: \Ham(S^2) \to \R,\] \[\tau_{k,k',B,B'} = c^0_{k,B} - c^0_{k',B'}\] is $2$-Lipschitz in Hofer's metric, is $C^0$-continuous, and extends to $G_{S^2}.$
\end{thm}
\begin{proof}
Let $U$ be the open disk of area $\min\{B,B'\}$ disjoint from \[\bb L^0_{k,k',B,B'} = \bb L^0_{k,B}  \cup \bb L^0_{k',B'}.\] Then by the Calabi property we obtain that $\tau = \tau_{k,k',B,B'}$ satisfies $\tau(\psi) = 0$ for all $\psi$ supported in $U.$ (Indeed, in this case it is easy to find a Hamiltonian $H$ supported in $U$ that generates $\psi.$) For $\eps > 0$ consider $\delta > 0$ provided by Lemma \ref{lma: CGHS2}. For $\theta \in \Ham(S^2)$ define the $C^0$-neighborhood $\cl{U}_{\theta,\delta} = \{ \theta \phi\;|\; d_{C^0}(\phi,\id) < \delta \}.$ Then for each $\theta \phi \in \cl{U}_{\theta,\delta}$ \[|\tau(\theta \phi) - \tau(\theta \psi)| \leq 2 \eps\] for $\psi$ supported in $U$ provided by Lemma \ref{lma: CGHS2}. However, in view of the controlled additivity property and the above vanishing property of $\tau,$ we have $\tau(\theta \psi) = \tau(\theta).$ Hence \[|\tau(\theta \phi) - \tau(\theta)| \leq 2 \eps,\] which proves the $C^0$-continuity of $\tau.$ Note that it proves more: in fact $\tau$ is uniformly continuous with respect to the uniform structure on $\Ham(S^2)$ given by the neighborhoods $\cl{V}_{\delta} = \{ (\theta,\psi)\;|\; d_{C^0}(\theta^{-1}\psi,\id) < \delta\},$ for $\delta > 0,$ of the diagonal in $\Ham(S^2) \times \Ham(S^2).$ Therefore $\tau$ extends to $G_{S^2}.$
\end{proof}




It is convenient to consider $\bb L^0_{1,1/2} = S \subset S^2$ to be the standard equator. Let $\tau_{k,B} = \tau_{k,1,B,1/2},$ where $c^0_{1,1/2}$ is the Lagrangian spectral invariant of $S \subset S^2.$ By Theorem \ref{thm-extends}, $\tau_{k,B}$ extends to $G_{S^2}^F$. We remark that with minor modifications the argument below also works for $\tau_{k,k',B,B'}$ with $k',B'$ fixed. 

\medskip
\noindent
{\it Proof of Theorem \ref{thm: non-simplicity}:}
Recall that we are given a smooth function $h:[-1/2,1/2) \to \R$  vanishing on $[-1/2,0].$ Consider $\phi \in G_{\bb D^2} \subset G_{S^2}$ generated by $H = h \circ z$.
Assume that a homeomorphism $\phi$, generated by the Hamiltonian $H$, lies in $G_{S^2}^F$, i.e., $\phi$ is the $C^0$-limit of Hamiltonian diffeomorphisms of Hofer's norm $\leq C$. Then
$|\tau_{k,B}(\phi)| \leq 2C$ for all $k$.
Observe now that by Lagrangian control and normalization properties,
for every rational $B \in (0,1/2),$ 
\begin{equation} \label{eq-Calabi-tau}
\left|\int_0^{1/2-B} h(s) ds \right| = (1-2B)\lim_{k \to +\infty} |\tau_{k,B}(\phi)| \leq 2C.
\end{equation}
Hence $h$ has bounded primitive.
\qed

\begin{qtn} {\rm Formula \eqref{eq-Calabi-tau} shows that one can reconstruct the Calabi invariant
of rotationally-symmetric Hamiltonian diffeomorphisms by using invariants $\tau_{k,B}$. Does
such a reconstruction exist for more general symplectomorphisms?
}
\end{qtn}

\color{black}

Now we prove the following consequence of Theorem \ref{thm: non-simplicity} (compare \cite{LeRoux-simplicity}). Recall first that, as explained in \cite[Proposition 2.2]{CGHS1} by a well-known argument of Epstein and Higman \cite{Epstein, Higman}, the group $G^F_{S^2}$ contains the commutator subgroup $[G_{S^2},G_{S^2}]$ of $G_{S^2}.$ Therefore the quotient group $Q = G_{S^2}/G^{F}_{S^2}$ is abelian. Hence every subgroup $H_0 \subset Q$ of $Q$ is normal and its preimage $H = \pi^{-1}(H_0)$ under the natural projection $\pi: G_{S^2} \to Q$ is a normal subgroup of $G_{S^2}$ containing $G^F_{S^2}.$ The results \cite{CGHS1,CGHS2} on the non-simplicity of $G_{S^2}$ are equivalent to the construction of various explicit embeddings $\R \hookrightarrow Q.$ Using our spectral invariants, we provide an explicit embedding of a large function space into $Q$ whose image contains all these copies of $\R.$
	



Let $\mathcal{G} \subset C^1((0,1/2])$ denote $C^1$-functions constant near $1/2,$ and let $\mathcal{G}_b$ denote bounded such functions. 

\begin{thm}\label{thm: subgroups}
The map $j: \mathcal{G} \to G_{S^2},$ $\rho \mapsto \phi,$ where $\phi$ is given by $H = h \circ z$ as above where $h(s) = -\rho'(1/2-s)$ for $s \in (0,1/2)$ and $h(s)=0$ for $s \leq 0,$ induces a monomorphism \[ \mathcal{G}/\mathcal{K} \hookrightarrow G_{S^2}/G^F_{S^2}\] for a subgroup $\cl K \subset \mathcal{G}_b.$ In particular, every subgroup of $\mathcal{G}/\mathcal{G}_b$ yields a normal subgroup of $G_{S^2}$ containing $G^F_{S^2}.$
\end{thm}


\begin{proof}

From the homomorphism $j: \cl G \to G_{S^2}$ we immediately obtain a homomorphism $j_1 = \pi \circ j: \cl G \to Q = G_{S^2}/G^F_{S^2}.$ It suffices to prove that $\cl K =  \Ker(j_1) \subset \cl G_b.$ Now if $\rho \in \Ker(j_1)$ then $\phi = j(\rho) \in G^F_{S^2}.$ By Theorem \ref{thm: non-simplicity} this implies that $\rho$ is bounded, that is, $\rho \in \cl G_b.$ \end{proof}

%
%
%
%

\begin{rmk}
It is easy to see by construction that direct analogues of Theorems \ref{thm: non-simplicity} and \ref{thm: subgroups} hold for $G_{\D^2}$ and $G_{\D^2}^F.$
\end{rmk}

\begin{rmk}
It would be interesting to compare $\cl K$ and $\cl G_b.$ See Remark \ref{rmk: discrepancy}.
\end{rmk}

\begin{rmk}
Finally, it is easy to see that for instance Theorem \ref{thm: main ker Cal} about the existence of flats in $\Ham(S^2)$ extends naturally to the groups $\mrm{Hameo}(S^2), \mrm{Hameo}_c(\bb D^2)$ from \cite{OhMuller} with their respective Hofer's metrics, and possibly further to $G_{S^2}^F, G_{\D^2}^F.$
\end{rmk}


\section{Lagrangian spectral invariants and estimators }\label{sec: prelim}

Here we construct Lagrangian estimators described in Section \ref{sec-estimators}
by using Lagrangian spectral invariants in symmetric products, and prove Theorems \ref{thm: main spec k} and \ref{thm: main k}. Interestingly enough, a remarkable Toeplitz tridiagonal
matrix, the $A_k$ Cartan matrix, naturally appears in the course of
our calculation of the critical points of the Landau-Ginzburg superpotential
combined from smooth and orbifold terms.

\subsection{Lagrangian Floer homology with bounding cochains and bulk deformation}

We briefly discuss the general algebraic properties of the Lagrangian Floer homology theory with weak bounding cochains and bulk. We refer to \cite{MakSmith-links} for a slightly more detailed discussion, and to the original work \cite{FO3:book-vol12,FO3-qm,FO3-toric-bulk, FO3-degeneration,FO3-toric-I} for all detailed definitions. We also remark that the Fukaya algebra of a Bohr-Sommerfeld Lagrangian submanifold in a rational symplectic manifold, that is when $[\om]$ is contained in the image of $H^2(M;\Q)$ inside $H^2(M;\R),$ was constructed in \cite{CharestWoodward-Fukaya} by classical transversality techniques. Hence one could feasibly carry out all constructions in this paper by the above techniques, when restricted to the rational setting, which is sufficient for our purposes. We also note that we establish our result as a rather formal consequence of the methods of \cite{MakSmith-links,FO3:book-vol12,FO3-qm}, hence it applies with whichever perturbation schemes these papers do.

For a subgroup $\Gamma \subset \R,$ define the Novikov field with coefficients in the field $\bK = \C$ as \[\Lambda_{\Gamma} = \{\sum_j a_j T^{\kappa_j}\,|\, a_j \in \bK,\, \kappa_j \in \Gamma,\, \kappa_j \to +\infty \}.\] This field possesses a non-Archimedean valuation $\nu: \Lambda_{\Gamma} \to \R \cup \{+\infty\}$ given by $\nu(0) = +\infty,$ and \[\nu(\sum a_j T^{\kappa_j}) = \min\{\kappa_j\,|\,a_j \neq 0 \}.\] For now we may assume that $\Gamma = \R,$ but later it will be important to choose a smaller subgroup. We often omit the subscript $\Gamma,$ and write $\Lambda$ for $\Lambda_{\Gamma}.$ Set $\Lambda_0 = \nu^{-1}([0,+\infty)) \subset \Lambda$ to be the subring of elements of non-negative valuation, and $\Lambda_{+} = \nu^{-1}((0,+\infty)) \subset \Lambda_0$ the ideal of elements of positive valuation.

Given a closed connected oriented spin Lagrangian submanifold $L \subset M,$ and an $\om$-tame almost complex structure on $M,$ considering the moduli spaces of $J$-holomorphic disks with boundary on $L$ and with boundary and interior punctures, and suitable virtual perturbations required to regularize the problem, as well as suitable homological perturbation techniques, yields the following maps. First, considering only $k+1$ boundary punctures, for $k \geq 0,$ we have the maps \[ m_k: H^*(L;\Lambda)^{\otimes k} \to H^*(L;\Lambda).\] These maps satisfy the relations of a curved filtered $A_{\infty}$ algebra. Furthermore, these maps decompose as $m_k = \sum m_{k,\beta} T^{\om(\beta)},$ where the sum runs over the relative homology classes $\beta \in H_2(M,L; \Z)$ of disks in $M$ with boundary on $L.$ Moreover, considering also $l$ interior punctures yields maps \[ q_{l,k}:H^*(M;\Lambda)^{\otimes l} \otimes H^*(L;\Lambda)^{\otimes k} \to H^*(L;\Lambda).\] Given a ``bulk" class $\mathbf{b} \in H_{\ast}(M; \Lambda_{+})$ we can deform the $m_k$ operations to: \[ m^{\bf b}_k(x_1 \otimes \ldots \otimes x_k) = \sum_{r\geq 0} q_{r,k}(\bf b^{\otimes r} \otimes x_1 \otimes \ldots \otimes x_k).\] The operations $m^{\bf b}_k$ for $k \geq 0$ also satisfy the relations of a curved filtered $A_{\infty}$ algebra, and decompose into a sum $m^{\bf b}_k = \sum m^{\bf b}_{k,\beta} T^{\om(\beta)}$ as above, over the relative homology classes $\beta \in H_2(M,L; \Z)$ of disks in $M$ with boundary on $L.$

Furthermore, given a ``cochain" class $b = b_0 + b_+$ for $b_0 \in H^1(L;\C)$ and $b_+ \in H^1(L;\Lambda_{+}),$ we define the $b$-twisted $A_{\infty}$ operations with bulk $\bf b \in H_*(M;\Lambda_{+})$ as follows: first set \[ m^{\bf b, b_0}_{k, \beta} = e^{\brat{b_0, \del \beta}} \,m^{\bf b}_{k, \beta}.\] These operations again satisfy the curved filtered $A_{\infty}$ equations. Now, to further deform by $b_+ \in H^1(L;\Lambda_{+}),$  we set \[ m^{\bf b, b}_{k, \beta}(x_1 \otimes \ldots \otimes x_k) = \sum_{l \geq 0} \sum m^{\bf b, b_0}_{k+l, \beta}(b_+ \otimes \ldots \otimes b_+ \otimes x_1 \otimes b_+ \otimes\ldots \otimes b_+ \otimes x_k \otimes b_+ \otimes \ldots \otimes b_+),\] where the interior sum runs over all possible positions of the $k$ symbols $x_1,\ldots,x_k$ appearing in this order, and $l$ symbols $b_+,$ and finally we set \[ m^{\bf b, b}_{k} = \sum_{\beta} T^{\om(\beta)} m^{\bf b, b}_{k, \beta}.\] The operations $m^{\bf b, b}_{k}$ also satisfy the curved filtered $A_{\infty}$ equations. Moreover, if $b_0 - b'_0 \in H^1(L; 2\pi \sqrt{-1}\Z),$ then  $m^{\bf b, b}_{k} = m^{\bf b, b'}_{k},$ $b = b_0 + b_+, b' = b'_0 + b_+.$

Given $\bf b \in H_*(M;\Lambda_{+}),$ we say that $b = b_0+b_+ \in H^1(L;\Lambda_0)$ is a {\em weak bounding cochain} for $\{m^{\bf b}_k\}$ if there exists a constant $c \in \Lambda_{+}$ such that \[ \sum_k m^{\bf b, b_0}_k(b_+ \otimes \ldots \otimes b_+) = c \cdot 1_L,\] where $1_L \in H^0(L;\Lambda_0)$ is the cohomological unit.  We call $c = W^{\bf b}(b)$ the potential function of the weak bounding cochain $b.$ Note that it depends only on the class of $b$ in $H^1(L;\Lambda_0)/H^1(L;2\pi\sqrt{-1} \Z),$ which is well-defined as there is no torsion in $H^1(L;\Z).$ We shall henceforth consider bounding cochains as elements of $H^1(L;\Lambda_0)/H^1(L;2\pi\sqrt{-1} \Z).$

Now we have the following result regarding Lagrangian tori, which was first proven in the context of Lagrangian torus fibers of toric manifolds in {\cite[Theorem 4.10]{FO3-toric-I}, \cite[Theorem 3.16]{FO3-toric-bulk}}.

\begin{thm}[{\cite[Theorem 2.3]{FO3-degeneration}}]\label{thm: crit W}
If $L \cong T^n$ is a Lagrangian torus and all elements of $H^1(L;\Lambda_0)/H^1(L;2\pi\sqrt{-1} \Z)$ are weak bounding cochains for $\{m^{\bf b}_k\},$ then if $b$ is a critical point of the potential function \[W^{\bf b}: H^1(L;\Lambda_0)/H^1(L;2\pi\sqrt{-1} \Z) \to \Lambda_{+}\] with $H^1(L;\Lambda_0)/H^1(L;2\pi\sqrt{-1} \Z) \cong (\Lambda_0/2\pi\sqrt{-1} \Z)^n$ identified with $\Lambda_0\setminus\Lambda_{+}$ by the exponential function, then $m^{\bf b, b}_1 = 0$ and hence the $(\bf b, b)$-deformed Floer cohomology of $L$ is isomorphic to $H^*(L;\Lambda).$ Moreover, this implies that $L$ is non-displaceable by Hamiltonian isotopies in $M.$
\end{thm}

Finally we note that $H^*(L;\Lambda) \cong H^*(L;\C) \otimes_{\C} \Lambda$ possesses a non-Archimedean filtration function $\cl{A}$ determined by requiring that the basis $E = E_0 \otimes 1,$ for $E_0 = (e_1,\ldots,e_B)$ a basis of $H^*(L;\C)$ be orthonormal in the sense that $\cl{A}(e_j \otimes 1) = 0$ for all $1\leq j \leq B,$ $\cl{A}(0) = -\infty,$ and for all $(\lambda_1,\ldots,\lambda_B) \in \Lambda^{B},$ \[\cl{A}(\sum \lambda_j e_j \otimes 1) = \max \{ \cl{A}(e_j \otimes 1) - \nu(\lambda_j)\}.\] It is important to note that the $A_{\infty}$ algebra $\{m_k^{\bf b, b}\}$ from Theorem \ref{thm: crit W} has unit $1_L \in H^*(L;\Lambda)$ of non-Archimedean filtration level $\cl{A}(1_L) = 0.$


\subsection{Lagrangian spectral invariants}\label{subsec:Lagrangian spectral}
In this section we discuss Lagrangian spectral invariants. While essentially going back to Viterbo \cite{Viterbo-specGF}, they were defined in Lagrangian Floer homology by a number of authors in varying degrees of generality, starting with Oh \cite{Oh-relative-cot}, Leclercq \cite{Leclercq-spectral}, and Monzner-Vichery-Zapolsky \cite{MonznerVicheryZapolsky}. However, the two main contributions that are relevant to our goals are the paper of Leclercq-Zapolsky \cite{LeclercqZapolsky} in the context of monotone Lagrangians, and of Fukaya-Oh-Ohta-Ono \cite[Definition 17.15]{FO3-qm} in the context of Floer homology with bounding cochains and bulk deformation.

\subsubsection{Discrete submonoids and their associated subgroups}

First we formulate the spaces of possible values of our spectral invariants. Following \cite{FO3-qm}, consider elements $\bf b \in H_*(M;\Lambda_0),$ $b \in H^1(L;\Lambda_0).$ We say that they are {\em gapped} if they can be written as \[ \bf b = \sum_{g \in G(\bf{b})} \bf b_g T^{g} ,\;\; \bf{b}_g \in H_*(M;\C) \] \[ b = \sum_{g \in G(b)} b_g T^{g},\;\; b_g \in H^1(L;\C)\] where $G(\bf b),$ $G(b)$ are discrete submonoids of $\R_{\geq 0}.$ In practice all relevant elements $\bf b, b$ will be gapped, so we assume that they are for the rest of this section.

Given an $\om$-tame almost complex structure $J$ on $M$ define the submonoid $G(L,\om,J)$ to be generated by the areas $\om([u])$ of all $J$-holomorphic disks $u$ with boundary on $L.$ It is discrete by Gromov compactness. Note that in the orbifold setting below, one includes both areas of smooth holomorphic disks and those of orbifold holomorphic disks.

\begin{df}\label{def: submonoid}
Let $G(L,\bf b,b) \subset \R_{\geq 0}$ be the discrete submonoid generated by the union $G(\bf b) \cup G(b) \cup G(L,\om,J).$ Furthermore let $\Gamma(\bf b),$ $\Gamma(b),$ $\Gamma(L,\om,J),$ $\Gamma(L, \bf b, b)$ be the {\em subgroups} of $\R$ generated by the monoids $G(\bf b) , G(b) , G(L,\om,J), G(L,\bf b,b)$ respectively. We call $(L,\bf b, b)$ {\em rational} if $\Gamma(L,\bf b, b)$ is a discrete subgroup of $\R.$
\end{df}

Given a Hamiltonian $H \in \sm{[0,1] \times M, \R}$ we consider the chords $x:[0,1] \to M$ with $x(0),x(1) \in L,$ satisfying $\dot{x}(t) = X^t_H(x(t))$ for all $t \in [0,1].$ Furthermore, we restrict attention to only those chords that are contractible relative to $L.$ Set $\spec(\LH)$ to be the set of all actions of pairs $(x,\ol{x})$ each consisting of a chord $x$ and its contraction $\ol{x}$ to $L,$ called a capping, $\overline{x}:\D \to M,$ $\overline{x}|_{\del \D \cap \{\Im(z) \geq 0\}} = x,$ $\overline{x}(\del \D \cap \{\Im(z) \leq 0\}) \subset L.$ The action of $(x, \ol{x})$ is given by \[\cA_{\LH}(x,\overline{x}) = \int_0^1 H(t,x(t))\,dt - \int_{\overline{x}} \om.\]

\begin{df}\label{def: deformed spec}
Define the $(\bf b, b)$-deformed spectrum $\mrm{Spec}(\LH,\bf b, b)$ of $\LH$ by \[ \mrm{Spec}(\LH,\bf b, b) = \spec(\LH) + \Gamma(L,\bf b, b).\] We recall that for two subsets $A,B$ of $\R,$ $A+B$ is the subset of $\R$ given by $A+B = \{ a+b \;|\; a\in A,\, b\in B \}.$
\end{df}

\subsubsection{Filtered Floer complex and spectral invariants}


Consider a Hamiltonian $H$ and $L \subset M$ a Lagrangian as above. Let $(\bfb)$ be a weak bounding cochain with bulk deformation. Assuming that $(\LH)$ is non-degenerate, that is $\phi^1_H(L)$ intersects $L$ transversely, and $\{J_t\}_{t \in [0,1]}$ a time-dependent $\om$-compatible almost complex structure on $M,$ following \cite{FO3:book-vol12,FO3-qm} one constructs a filtered finite rank $\Lambda_0$-complex $CF(\LH,\bfb; \Lambda_0),$ where $\Lambda = \Lambda_{\Gamma}$ for $\Gamma = \Gamma(L,\bf b, b).$ 


We set $CF(\LH,\bfb) = CF(\LH,\bfb;\Lambda_0) \otimes_{\Lambda_0} \Lambda.$ The module $CF(\LH,\bfb)$ comes with an non-Archimedean filtration function $\cl{A}_{\LH}$ whose values are contained in $\spec(\LH,\bf b, b) \cup \{-\infty\}.$ Indeed, as a $\Lambda$-module $CF(\LH,\bfb)$ is given by the completion with respect to the action functional of the vector space generated by pairs $(x,\ol{x})$ of Hamiltonian $H$-chords from $L$ to $L$ contractible relative to $L,$ where we identify between $(x,\ol{x}),$ $(x,\ol{x}')$ if $v = \ol{x}' \# \ol{x}^{-}: (\D,\del \D) \to (M,L),$ defined by gluing suitably reparametrized $\ol{x}'$ and $\ol{x}^{-}(z) = \ol{x}(\ol{z})$ along their common boundary chord, satisfies $\brat{[\om],[v]} = 0.$ In this case the filtration $\cl{A}$ is defined by declaring any $\Lambda$-basis $[(x_i,\ol{x}_i)],$ for $\{x_i\}$ the finite set of contractible $H$-chords from $L$ to $L$ an orthogonal $\Lambda$-basis of $CF(\LH,\bfb).$ Finally, the homology $HF(\LH,\bfb)$ of $CF(\LH,\bfb)$ is naturally isomorphic to self-Floer homology $HF(L,\bfb) = HF((L,\bfb),(L,\bfb))$ of $L$ deformed by $(\bfb).$ We write $\Phi_H: HF(L,\bfb) \to HF(\LH,\bfb)$ for this isomorphism. In the setting of Theorem \ref{thm: crit W} there is a natural isomorphism between $HF(\LH,\bfb)$ and $H^*(L,\Lambda),$ which in particular explains why $L$ is not Hamiltonianly displaceable: if $\phi^1_H(L) \cap L = \emptyset$ then $CF(\LH,\bfb) = 0$ and hence one must have $HF(\LH,\bfb) = 0.$

For $a \in \R$ the subspace $CF(\LH,\bfb)^a$ generated over $\Lambda_0$ by all $[(x,\bar{x})]$ satisfying $\cA_{\LH}(x,\bar{x}) < a$ forms a subcomplex of $CF(\LH,\bfb).$ We denote its homology by $HF(\LH,\bfb)^a.$ It comes with a natural map $HF(\LH,\bfb)^a \to HF(\LH,\bfb).$

Given a class $z \in HF(L,\bfb),$ we define its spectral invariant with respect to $H$ as in \cite[Definition 17.15]{FO3-qm} by \[ c(L,\bfb; z, H) = \inf \{a \in \R\;|\; \Phi_H(z) \in \ima(HF(\LH,\bfb)^a \to HF(\LH,\bfb))\}.\]

In the sequel we will primarily work with $z = 1_L,$ the unit of the algebra on $HF(L,\bfb)$ induced by the $m_2^{\bfb}$ operation. Following the arguments of \cite[Theorem 7.2]{FO3-qm}, \cite[Theorem 35]{LeclercqZapolsky} together with the argument in \cite[Remark 4.3.2]{PolterovichRosen} for rational spectrality, it is straightforward to show the following properties of the spectral invariant $c(L,\bfb; z, H)$:

\begin{enumerate}
	\item {\em non-degenerate spectrality:} for each $z \in HF(L,\bfb) \setminus \{0\},$ and Hamiltonian $H$ such that $\LH$ is non-degenerate, \[c(L,\bfb;z,H) \in \spec(\LH,\bfb).\]
	\item {\em rational spectrality:} if $\Gamma(H,\bfb)$ is rational, then for $z \in HF(L,\bfb) \setminus \{0\},$ the condition \[c(L,\bfb;z,H) \in \spec(\LH,\bfb)\] holds for {\em each} Hamiltonian $H,$
	\item {\em non-Archimedean property:} for each Hamiltonian $H,$ $c(L,\bfb;-,H)$ is a non-Archimedean filtration function on $HF(L,\bfb)$ as a module over the Novikov field $\Lambda$ with its natural valuation.
	\item {\em Hofer-Lipschitz:} for each $z \in HF(L,\bfb) \setminus \{0\},$ \[|c(L,\bfb;z,F) - c(L,\bfb;z,G)| \leq \max\{\cE_{+}(F-G), \cE_{-}(F-G)\}\] where for a Hamiltonian $H,$ we set $\cE_{+}(H) = \int_{0}^{1} \max_M(H_t)\,dt$ and $\cE_{-}(H) = \cE_{+}(-H).$ Note that $\cE_{\pm}(H) \leq \int_0^1 \max_M |H_t|\,dt.$
	\item {\em normalization:} For each $H \in \sm{[0,1] \times M, \R}$ and $b\in \sm{[0,1],\R},$ \[c(L,\bfb;z,H+b) = c(L,\bfb;z,H) + \int_{0}^{1} b(t)\,dt.\]

	\item {\em monotonicity:} if two Hamiltonians $F,G$ satisfy $F_t \leq G_t$ for all $t \in [0,1],$ then $c(L,\bfb; z, F) \leq c(L,\bfb; z, G)$ for each $z \in HF(L,\bfb)$
	\item {\em Lagrangian control:} if $\Gamma(H,\bfb)$ is rational and $(H_t)|_L = c(t) \in \mathbb{R}$ for all $t \in [0,1]$ then setting $c_+(H) = c(L,\bfb; 1_L, H)$ we have \[c_{+}(H)=\int_0^1 c(t) \,dt\,\]
	hence for all $H \in \sm{[0,1] \times M,\R},$ $\;\int_0^1 \min_L H_t \,dt \leq c_{+}(H) \leq \int_0^1 \max_L H_t \,dt.$
	\item {\em homotopy invariance:} for $H_t$ of mean-zero for all $t \in [0,1],$ $c(L,\bfb;z,H)$ depends only on the class $\til{\phi}_H \in \til{\Ham}(M,\om)$ of the Hamiltonian isotopy $\{\phi^t_H\}_{t \in [0,1]},$
	\item {\em triangle inequality:} for each $z,w \in HF(L,\bfb),$ and Hamiltonians $F,G,$ \[c(L,\bfb; m^{\bfb}_2(z,w),F\#G) \leq c(L,\bfb;z ,F) + c(L,\bfb;w ,G),\] where $F \# G(t,x) = F(t,x) + G(t,(\phi^t_F)^{-1}x)$ generates the flow $\{ \phi^t_F \phi^t_G\}_{t \in [0,1]}.$ \end{enumerate}

\subsection{Orbifold setting}

We shall consider only a very simple kind of a symplectic orbifold $X.$ It is called a global quotient symplectic
orbifold, and consists of the data of a closed symplectic manifold $\til{M}$ and an effective symplectic action of a finite group $G$ on it.

In fact, we shall only consider $\til{M} = M^k = M \times \ldots \times M$ for a symplectic manifold $M$ and the symmetric group $G = \Sym_k$ acting on $\til{M}$ by permutations of the coordinates. The corresponding global quotient orbifold is called the symmetric power $X = Sym^k(M)$ of $M.$

The inertia orbifold $IX$ of a global quotient orbifold $X$ is itself a global quotient orbifold given by the action of $G$ on the disjoint union $\sqcup_{g \in G} \til{M}^g,$ where for $g\in G,$ $\til{M}^g$ is the fixed point submanifold of $g,$ and $f \in G$ acts by $\til{M}^g \to \til{M}^{fgf^{-1}},$ $x \mapsto fx.$ See \cite{Adem-book} for further details. 

\subsubsection{Orbifold Hofer metric}
While on a symplectic orbifold one can define smooth functions, and hence Hamiltonian diffeomorphisms, in the case of a global quotient orbifold $X,$ the data of a smooth Hamiltonian $H \in \sm{[0,1] \times X, \R}$ is equivalent to the data of a $G$-invariant Hamiltonian $\til{H} \in  \sm{[0,1] \times \til{M}, \R},$ that is $H(t,g\cdot x) = H(t,x)$ for all $t \in [0,1], x\in \til{M}, g\in G.$ In our particular case $\til{M} = M^k$ and our Hamiltonians satisfy $H(t,x_1,\ldots, x_k) = H(t,x_{\sigma^{-1}(1)},\ldots, x_{\sigma^{-1}(k)})$ for all $t \in [0,1], x_1,\ldots,x_k \in M,$ and $\sigma \in \Sym_k.$ Given an orbifold Hamiltonian diffeomorphism $\phi$ of $X,$ which is equivalently a $G$-equivariant Hamiltonian diffeomorphism $\til{\phi}$ of $\til{M}$ that is generated by a $G$-invariant Hamiltonian $H$ on $\til{M},$ we define its {\em orbifold Hofer distance} to the identity by \[ d_{\mrm{Hofer}}(\phi, \id) = \inf_{\phi_H^1 = \til{\phi}} \intoi \max_{\til{M}} H(t,-) - \min_{\til{M}} H(t,-) \; dt \] where the infinimum runs over all such $G$-invariant Hamiltonians generating $\til{\phi}.$ Finally we remark that orbifold Hamiltonian diffeomorphisms form a group $\Ham(X),$ it is isomorphic to the identity component $\Ham(\til{M})^G$ of the subgroup of $\Ham(\til{M})$ consisting of $G$-equivariant Hamiltonian diffeomorphisms, and $d_{\mrm{Hofer}}$ extends to a bi-invariant (non-degenerate) metric on $\Ham(X).$ We denote by $\til{\Ham}(X)$ the universal cover of this group with basepoint at the identity. This is also a group.

\subsubsection{Orbifold Lagrangian Floer homology with bulk, and spectral invariants}\label{subsec: Lagr spectral}

%
%

It was explained in \cite{ChoPoddar} and summarized in \cite{MakSmith-links} that the above setup of Lagrangian Floer homology with weak bounding cochain and bulk deformation $\bfb$ holds in the setting of a smooth Lagrangian submanifold $L$ in the regular locus of a closed effective symplectic orbifold $X$ (which for us will be a global quotient orbifold $\til{M}/G,$ and whose regular locus is $\til{M}^0/G$ where $\til{M}^0$ is the set of points with trivial stabilizer) with the additional feature that we may consider bulk deformations by classes in the homology $H_*(IX;\Lambda_{+})$ of the inertia orbifold of $X$ by counting holomorphic disks with possible orbifold singularities at the interior punctures. Furthermore, the above constructions of spectral invariants and arguments proving their properties go through in this situation.

\subsubsection{Lagrangians in symmetric products and their spectral invariants}\label{subsubsec: Mak-Smith}
Consider $S^2$ with the symplectic form $\om$ of total area $1.$ Under this normalization consider the Lagrangian configuration $L^{0,j}_{2,B},$ $j = 0,1$, see \eqref{eq-lzero} above.
In \cite{MakSmith-links} Mak and Smith show, as translated to our normalizations, that if \[M = (S^2 \times S^2, \om \oplus 2a \om),\] with $0 < a < 3B - 1 = B-C,$ then the Lagrangian link $\mathbb{L}_{2,B} = \mathbb{L}^0_{2,B} \times S^1$ in $M$ is Hamiltonianly non-displaceable: for all $\phi \in \Ham(X),$ $\phi(\mathbb{L}_{2,B}) \cap \mathbb{L}_{2,B} \neq \emptyset.$ 

The key observation of this paper is that in fact the proof of \cite{MakSmith-links} gives strictly stronger information. We recall their approach. The two connected components of $\mathbb{L}_{2,B}$ are $L^j_{2,B} = L^{0,j}_{2,B} \times S$ where $S \subset S^2(2a)$ is the equator. Consider the product Lagrangian \[\cl{L}'_{2,B} = {L}^0_{2,B} \times L^1_{2,B}  \subset \til{M} = M \times M.\] Since $\cl{L}'_{2,B}$ is disjoint from the diagonal $\Delta_M \subset M \times M,$ it descends to a smooth Lagrangian submanifold $\cl{L}_{2,B}$ in the regular locus ${X}^{\rm{reg}}$ of the symplectic global quotient orbifold ${X} = (M \times M)/(\Z/2\Z),$ where $\Z/2\Z$ acts by exchanging the factors. Similarly, in the situation of configurations for arbitrary values $(k,B)$ of parameters, one defines \[\cl{L}'_{k,B} = {L}^0_{k,B} \times \ldots \times L^{k-1}_{k,B} \subset M^k,\] and it descends to a Lagrangian $\cl{L}_{k,B}$ in the regular locus of $X^{\mrm{reg}}$ the symplectic orbifold $X = Sym^k(M) = M^k/\Sym_k,$ where the symmetric group $\Sym_k$ acts on $M^k$ by permutations of the coordinates.




In \cite{MakSmith-links} the authors have proved for $k=2,$ and we show in Section \ref{sec:proof of orbifold homology} that the same extends to arbitrary values of parameters $(k,B)$ and $a < \frac{(k+1)B -1}{k-1}$ the following statement.

\begin{thm}\label{thm: orbifold homology}
There exists an integrable almost complex structure $J_M$ on $M$ such that with respect to the complex structure $J = Sym^k(J_M)$ on $X = Sym^k(M),$ the Lagrangian  $\cl{L}_{k,B} \subset X$  has a well-defined Fukaya algebra in the sense of \cite{FO3:book-vol12,FO3-qm,ChoPoddar}. Moreover, this Fukaya algebra admits gapped orbifold bulk deformation $\mathbf{b}$ and weak bounding cochain $b$ with the bulk-deformed Floer homology \[HF( (\cl{L}_{k,B},\mathbf{b},b), (\cl{L}_{k,B},\mathbf{b},b)) \cong HF( \cl{L}_{k,B}, \bfb )\] of $(\cl{L}_{k,B},\mathbf{b},b)$ with itself well-defined and isomorphic to $H^*(\cl{L}_{k,B};\Lambda)$ with coefficients in the Novikov field $\Lambda = \Lambda_{\Gamma},$ where $\Gamma = \Gamma(L,\bf b, b).$ Moreover, for $B \in (1/(k+1),1/2),$ $a \in (0,\frac{(k+1)B -1}{k-1})$ rational, $\bfb$ can be chosen in such a way that $\Gamma$ is rational. Furthermore, $HF(\cl{L}_{k,B},\bfb)$ is an associative unital algebra, and we denote its unit by $1_{\cl{L}_{k,B}}.$
\end{thm}

The proof of this theorem appears in Section \ref{sec:proof of orbifold homology} below. It implies in particular that the unit $1_{\cl{L}_{k,B}} \in HF(\cl{L}_{k,B},\mathbf{b},b)$ does not vanish. This enables us to use spectral invariants associated to this element. For $H\in \sm{[0,1] \times X, \R}$ we set \[ \sigma_{k,B}(H) = c(\cl{L}_{k,B},\bfb; 1_{\cl{L}_{k,B}}, H).\] Then $\sigma_{k,B}$ satisfies the properties of Lagrangian spectral invariants, including rational spectrality, from Section \ref{subsec:Lagrangian spectral}:

\begin{enumerate}
\item {\em spectrality:} $\Gamma(H,\bfb)$ being rational, the condition \[\sigma_{k,B}(H) \in \spec(\LH,\bfb)\] holds for {\em each} Hamiltonian $H.$
\item {\em Hofer-Lipschitz:} \[|\sigma_{k,B}(F) - \sigma_{k,B}(G)| \leq \max\{\cE_{+}(F-G), \cE_{-}(F-G)\}.\]

	\item {\em normalization:} For each $H \in \sm{[0,1] \times M, \R}$ and $b\in \sm{[0,1],\R},$ \[\sigma_{k,B}(H+b) = \sigma_{k,B}(H) + \int_{0}^{1} b(t)\,dt.\]

\item {\em monotonicity:} if two Hamiltonians $F,G$ satisfy $F_t \leq G_t$ for all $t \in [0,1],$ then $\sigma_{k,B}(F) \leq \sigma_{k,B}(G).$

\item {\em Lagrangian control:} $\Gamma(H,\bfb)$ being rational, if $(H_t)|_{\cl{L}_{k,B}} = c(t) \in \mathbb{R}$ for all $t \in [0,1],$ then we have \[\sigma_{k,B}(H)=\int_0^1 c(t) \,dt,\] hence for each Hamiltonian $H,$ $\;\int_0^1 \min_{\cl{L}_{k,B}} H_t \,dt \leq \sigma_{k,B}(H) \leq \int_0^1 \max_{\cl{L}_{k,B}} H_t \,dt.$

\item {\em homotopy invariance:} for $H_t$ of mean-zero for all $t \in [0,1],$ $\sigma_{k,B}(H)$ depends only on the class ${\phi}_H \in \til{\Ham}(X,\om)$ of the Hamiltonian isotopy $\{\phi^t_H\}_{t \in [0,1]},$
\item {\em triangle inequality:} for each two Hamiltonians $F,G,$ \[\sigma_{k,B}(F\#G) \leq \sigma_{k,B}(F) + \sigma_{k,B}(G).\] \end{enumerate}

\subsubsection{Proof of Theorems \ref{thm: main spec k} and \ref{thm: main k}}\label{subsubsec: proof A}

We now prove Theorems \ref{thm: main spec k} and \ref{thm: main k}. For a Hamiltonian $F \in \sm{M_a,\R}$ let $H \in \sm{X,\R}$ be the Hamiltonian on $X$ determined by the $\Sym_k$-invariant Hamiltonian $\til{H} = F \oplus \ldots \oplus F$ in $\sm{[0,1] \times M^k, \R},$ that is $\til{H}(t,x_1,\ldots, x_k) = F(t,x_1) + \ldots + F(t,x_k).$ Observe that if $F$ is mean-zero, then so is $H.$ Furthermore, if $F_1 \mapsto H_1,$ $F_2 \mapsto H_2,$ then $F_1 \# F_2 \mapsto H_1 \# H_2$ and the map $F \mapsto H$ induces a map $\til{\Ham}(M_a) \to \til{\Ham}(X).$

We set \[ c_{k,B}(F) = \frac{1}{k} \sigma_{k,B}(H),\] \[\mu_{k,B}(F) = \lim_{m\to \infty} \frac{1}{m} c_{k,B}(F^{\#m}).\]  The limit exists by Fekete's lemma by the subadditivity property of $\sigma_{k,B}.$ The Hofer-Lipschitz property of $c_{k,B}$ and $\mu_{k,B}$ follows from that of $\sigma_{k,B}$ by the subadditivity of Hofer's energy functional. The monotonicity, normalization, Lagrangian control, and independence of Hamiltonian properties are immediate consequences of those for $\sigma_{k,B}.$ Note the factor $1/k$ in the definition of $c_{k,B}:$ it serves for instance to obtain the Hofer-Lipschitz property with coefficient $1,$ as $\cl{E}_+(H) \leq k \cl{E}_+(F)$ for $F \mapsto H.$ Similarly, if $F_t|_{L^j_{k,B}} \equiv c_j(t),$ then $H_t|_{\cl{L}_{k,B}} \equiv \sum_{0 \leq j < k} c_j(t).$

Subadditivity for $c_{k,B}$ and commutative subadditivity for $\mu_{k,B}$ are direct consequences of the subadditivity of $\sigma_{k,B}.$ Positive homogeneity of $\mu_{k,B}$ is immediate from the definition of $\mu_{k,B}.$

It remains to prove the conjugation invariance of $\mu_{k,B},$ controlled additivity of $c_{k,B}$ and the Calabi properties of $c_{k,B}$ and $\mu_{k,B}.$

To prove conjugation invariance of $\mu_{k,B}$, we note that by the subadditivity of $\sigma_{k,B}$ and hence of $c_{k,B},$ we have for all $m \in \Z_{>0},$ \begin{equation}\label{ineq: conjugation} c_{k,B}(\phi^m) - q_{k,B}(\psi) \leq c_{k,B}(\psi\phi^m\psi^{-1}) \leq c_{k,B}(\phi^m) + q_{k,B}(\psi),\end{equation} where $q_{k,B}(\psi) = c_{k,B}(\psi) + c_{k,B}(\psi^{-1}).$ Hence dividing by $m$ and taking the limit as $m \to \infty,$ we obtain \[ \mu_{k,B}(\phi) = \mu_{k,B}(\psi\phi\psi^{-1})\] as required.

Now we prove the Calabi property. The proofs for $c_{k,B}$ and $\mu_{k,B}$ are identical, so we focus on $\mu_{k,B}$ for example. Suppose that $U \subset M_a$ is disjoint from $\mathbb{L}_{k,B}.$ Given a Hamiltonian $H \in \sm{[0,1] \times M,\R}$ supported in $[0,1] \times U,$ consider the Hamiltonian \begin{equation}\label{eq: extension by zero} G(t,x) = H(t,x) + b(t)\end{equation} such that for all $t \in [0,1],$ \[ b(t) = - \frac{1}{\vol(M_a)} \int_U H(t,x)\, \om^2.\] In that case $\mu_{k,B}(H) = 0.$ Hence in view of the normalization property of $\mu_{k,B}$ we obtain \[ \mu_{k,B}(G) = - \int_{0}^1 b(t)\,dt = - \frac{1}{\vol(M_a)} \int_{0}^1 \int_U H(t,x)\, \om^2\,dt = - \frac{1}{\vol(M_a)} \Cal_U(\{\phi^t_H\}).\] 

Finally, to prove controlled additivity it suffices to observe that for $H$ supported in $[0,1] \times U$ and each Hamiltonian $F \in \sm{[0,1] \times M_a, \R},$ there is a bijective correspondence between the generators of the Lagrangian Floer complex of $F \# H$ and that of $F,$ and furthermore \[\spec(\cl{L}_{k,B},\bfb;1_{\cl{L}_{k,b}},\til{F\#H}) = \spec(\cl{L}_{k,B},\bfb;1_{\cl{L}_{k,b}},\til{F}).\] The proof now proceeds by the spectrality and Hofer-Lipschitz axioms for the family $\til{F \# sH}$ of Hamiltonians where $s \in [0,1]$, as that of the Lagrangian control property. We obtain that $c_{k,B}(F \# sH)$ is constant in $s \in [0,1]$ and hence $c_{k,B}(F\#H) = c_{k,B}(F).$ The proof is now concluded by the normalization axiom to pass to the mean-zero Hamiltonian $G$ as in \eqref{eq: extension by zero} instead of $H.$

\qed


\subsubsection{Proof of Theorem \ref{thm: orbifold homology}}\label{sec:proof of orbifold homology}

We explain how the Floer-theoretical approach of Mak-Smith \cite[Remark 1.10]{MakSmith-links} applies to multiple level sets of the moment map $z:S^2 \to [-1/2,1/2].$ This allows us to prove Theorem \ref{thm: orbifold homology}. 

Consider our Lagrangian configuration $\mathbb{L}^0_{k,B} \subset S^2$ of $k$ Lagrangians on $S^2.$ For simplicity let us shorten the notation: $L^{0,i}_{k,B} = L_i,$ $\mathbb{L}^0_{k,B} = \mathbb{L},$ $i \in I = \{ 0,\ldots, k-1 \}.$

%

Now considering the manifolds $L'_i = L_i \times S$ in $M = M_a = S^2 \times S^2(2a)$ for $0< a < B-C,$ where $S\subset S^2(2a)$ is the equator, we obtain a Lagrangian submanifold $\cl{L}'_{k,B}$ of $\til{M} = M^k$ and a Lagrangian submanifold $\cl{L}_{k,B}$ of $X = \Sym^k(M).$ Recall that $X$ is the global quotient orbifold $X = \til{M}/\Sym_k$ where $\Sym_k$ denotes the symmetric group on $k$ elements. Set also $\mathbb{L}' = \sqcup_{i\in I} L'_i.$

We consider the orbifold Lagrangian Floer homology of $\cl{L}_{k,B}$ in $X$ with bulk deformation and bounding cochain $(\bfb)$ of a special form. As in \cite{MakSmith-links}, the bulk deformation will take the form \[\bf b = \beta_{orb} [X_{\gamma}] + \bf b_{smooth}\] for $\bf b_{smooth} \in H_{4k-2}(X; \Lambda_{+}),$ $\beta_{orb} \in \Lambda_{+},$ $\gamma$ a conjugacy class in $G = \Sym_k,$ and $X_{\gamma} \cong \til{M}^g/C(g),$ for any $g \in \gamma$ and $C(g)$ its centralizer, the connected component of the inertia orbifold $IX$ of $X$ corresponding to $\gamma.$ Of course $[X_{\gamma}]$ is the fundamental class of $X_{\gamma}.$

We make the following choice of $\gamma$ which is important for our particular situation: $\gamma$ is the conjugacy class of transpositions in $G= \Sym_k,$ that is, the permutations of type $(1)^{k-2}(2).$ Note that each element of $\gamma$ is of order $2.$ It is also easy to calculate that $X_{\gamma} \cong M \times \Sym^{k-2} M.$


Recall, following \cite{MakSmith-links,OszvathSzabo-symm} that a map $u:S \to (X,\cl{L}_{k,B}),$ for a stable disk $S,$ corresponds ``tautologically" to a map $\nu: \Sigma \to (M,\mathbb{L}_{k,B}),$ where $\pi_{\Sigma}:\Sigma \to S$ is a degree $k$ branched cover: $u$ is recovered from $\nu$ by $u(z) = \nu(\pi_{\Sigma}^{-1}(z)) \in X = Sym^k(M).$ Vice versa, $\nu$ is obtained as follows: first consider the fiber product of $u$ and the projection $p: M^k \to X.$ This yields a $\Sym_k$-equivariant branched cover $\til{\Sigma} \to S$ and a $\Sym_k$-equivariant holomorphic map $V: \til{\Sigma} \to M^k.$ Consider $V_1 = \pi_1 \circ V: \til{\Sigma} \to M,$ where $\pi_1: M^k \to M$ is the projection to the first coordinate. This map is invariant under the action of the stabilizer $G_1 \cong \Sym_{k-1}$ of $1 \in \{1,\ldots,k\}$ under the action of $\Sym_k.$ We then define $\pi_{\Sigma}:\Sigma \to S$ to be the quotient $\til{\Sigma}/G_1$ by the action of $G_1$ with the induced projection to $S$ and the induced map $V:\Sigma \to M.$ It is then easy to see that the boundary components correspond. Furthermore the $\om_X$-area of $u$ coincides with the $\om_M$-area of $\nu.$ In particular, following the argument of \cite[Lemma 3.2]{MakSmith-links}, one can show the following.

\begin{lma}\label{lma: connected taut}
Let $\nu:(\Sigma, \del \Sigma) \to (M, \mathbb{L}')$ be obtained by tautological correspondence from a stable disk $u:(S,\del S) \to (X,\cl{L}_{k,B}).$ Then $\del \Sigma$ has $k$ connected components. If $\Sigma$ has $k$ connected components, then it consists of $k$ disks. If $\Sigma$ has $k-1$ connected components, then it consists of $k-2$ disks and one curve with $2$ boundary components.
\end{lma}

{We remark that of course $\Sigma$ could have less than $k-1$ connected components. However, we shall not require knowledge of the topology of $\Sigma$ in this case. Similarly, the curve with two boundary components obtained in the second case of Lemma \ref{lma: connected taut} might or might not be an annulus. However, only the annuli will contribute to the lowest order term of the superpotential.} For more details about the tautological correspondence we refer to \cite{MakSmith-links,OszvathSzabo-symm}.

Following the dimension calculations in \cite[Section 3.1]{MakSmith-links} we obtain the following analogue of \cite[Lemma 3.12]{MakSmith-links}.

\begin{lma}\label{lma: proportional to 1}
	{Suppose that for each $b=b_0+b_+ \in H^1(\cL_{k,B};\Lambda_0),$ each non-constant $J_{X}$-holomorphic orbifold stable map $u$ to $(X,\cL_{k,B})$ that contributes to the sum \begin{equation}\label{eq: sum in proportional to 1} \sum_k m_k^{\bf b, b_0}(b_+ \otimes \ldots \otimes b_+),\end{equation} tautologically corresponds to a $J_M$-holomorphic curve $\nu:(\Sigma,\del \Sigma) \to (M, \bb L')$ with non-zero Maslov index $\mu(\nu)$. Then \[H^1(\cL_{k,B};\Lambda_0)/H^1(\cL_{k,B};2\pi \sqrt{-1} \Z)\] consists of weak bounding cochains. Moreover, for each such stable map $u,$ the corresponding $\nu$ satisfies $\mu(\nu) = 2.$}
\end{lma}

 \medskip
 \noindent
 {\sc Dimension counting.} For the reader's convenience, we briefly outline dimension
 counting and thus outline the proof of Lemma \ref{lma: proportional to 1}. It would be convenient to analyze a slightly more general case
 when we have a pseudo-holomorphic map $\nu: (\Sigma, \partial \Sigma) \to (N,Q)$,
 where $Q = Q_1 \sqcup \ldots \sqcup Q_k$ is an $n$-dimensional Lagrangian submanifold of a symplectic manifold $N$.
 As it was explained above, this curve together with a $k$-fold branched cover $\Sigma \to S$,
 where $S$ is a disc,  corresponds to an orbifold disc $u:(S, \partial S) \to (X, \cl K)$ in the $k$-fold symmetric product with the boundary on $\mathcal{K}:= (Q_1 \times \ldots \times Q_k)/\text{Sym}_k$. (In our particular application, $N = M_a,$ $Q = \bb L_{k,B}$ and $\mathcal{K} = \cl L_{k,B}.$)

Here $\Sigma$ is a Riemann surface with $k$ boundary components obtained as a degree $k$ branched cover of the disc $S$ with $\ell_2$ branch points of order $2$.
Thus the total space of $\Sigma$'s, automorphisms taken into account, has dimension $2\ell_2 - 3$, where we used that $\dim PSL(2,\mathbb{R})=3$. Let us emphasize that these $\ell_2$ points automatically map by $u$ to our component $X_{\gamma}$ of the inertia orbifold, constrained by orbifold bulk of codimension $0$.

The space of parameterized holomorphic maps $\nu$ of $\Sigma$
with $n$-dimensional Lagrangian boundary has dimension  $n\chi(\Sigma) + \mu$, where $\mu$ stands for the Maslov index of $\nu$. 

Additionally, $S$ is equipped with $r$ boundary marked points which go to the bounding cochain in $\cl K$ of codimension $1$, and with $\ell_1$ interior marked points which go to the smooth bulk of codimension $2$ in $X.$ Finally, we have an extra output boundary point, $w \in \partial \Sigma$.

Combining this, we get that the virtual dimension  of moduli space $\mathcal{M}$ of such curves $u$ equals
 $$\dim \mathcal{M} = n\chi(\Sigma) + \mu + (r+1 + 2\ell_1) + (2\ell_2-3) - r - 2\ell_1\;,$$
where $\chi$ stands for the Euler characteristic. Furthermore, by the Riemann-Hurwitz formula, $\chi(\Sigma) = k -\ell_2$. Finally, $\dim \mathcal{K} = nk$. It follows that
 $$
 \dim \mathcal{M} = \dim \mathcal{K} + (2-n)(k-\chi(\Sigma)) + \mu-2\;.
 $$

Let us compute the Maslov index of curves $\nu$ corresponding to $u$ which contribute to the coefficient of the unit $1_{\cl K} \in H^*(\cl K; \Lambda_0)$ in the sum as in \eqref{eq: sum in proportional to 1}. We need the evaluation map $u \mapsto u(w)$ from $\mathcal{M}$ to $\mathcal{K}$ to have non-zero degree, and in particular $\dim \mathcal{M}  = \dim \mathcal{K}$. Thus we get
\begin{equation} \label{eq-dim-1}
 (2-n)(k-\chi(\Sigma)) + \mu =2\;.
\end{equation}

This finishes the general discussion of the dimension count. Let us now apply it in two special cases relevant to us. In both cases, we use a toric almost complex structure and \cite[Lemma 3.24]{MakSmith-links} in order to show that the Maslov index is non-negative for the holomorphic curves in question.

When $N=S^2$, i.e., when we work directly on $S^2$ without any stabilization, we have $n=1$ so equation \eqref{eq-dim-1} holds\footnote{We thank Cheuk Yu Mak for this clarification} for an annulus $\Sigma$ with Maslov $0.$ In fact, writing for a general curve $u$ contributing to the sum that $\ell_2 = k - \chi(\Sigma),$ and noting that $\mu$ is non-negative and even in our situation, we see that the only two possibilities are $\ell_2 = 0$ and $\mu = 2,$ which corresponds to a smooth disk, and $\ell_2 = 2$ and $\mu = 0,$ which corresponds to an annulus.

When $N= S^2 \times S^2(2a)$, we have that $n=2$ and the term involving $\Sigma$ in formula \eqref{eq-dim-1} vanishes, giving $\mu=2$. In fact, using the more general fact that if $u$ contributes to the sum, we must have $\mu = (2-n)(k-\chi(\Sigma)) + \mu \leq 2,$ and $\mu$ being non-negative and even, implies that $\mu = 2$ or $\mu = 0.$ If the last option is impossible, we obtain $\mu = 2.$

\medskip

\color{black}


\blue{Furthermore, one can prove following \cite[Lemma 3.25]{MakSmith-links} that there exists an almost complex structure $J_M$ on $M$ such that the Maslov non-zero condition in Lemma \ref{lma: proportional to 1} is satisfied. We now outline how this is achieved. 

When setting the Lagrangian Floer theory with bulk, and calculating the superpotential, we perform the Fukaya trick to obtain the following situation. From the point of view of complex structures one counts holomorphic curves with boundaries on $\mathbb{K} = \sqcup_i K_i$ with $K_i = R_i \times S_i$ instead of $L'_i = L_i \times S$ where $R_i$ and $S_i$ are small perturbations of $L_i$ and $S$ of the form $z^{-1}(\gamma_i)$ and $z^{-1}(\delta_i)$ for $\delta_i \in \R$ small with $\gamma_i<\gamma_j, \delta_i < \delta_j$ whenever $i<j$. From the point of view of symplectic areas of these holomorphic curves, they are still counted with respect to the $L'_i.$ Indeed, the correspondence between curves with boundary on $\bb L'$ and those with boundary on $\bb K$ is carried out by a suitable $C^1$-small diffeomorphism of $M_a$ carrying $L'_i$ to $K_i$ for each $i.$}

Identifying $S^2 \times S^2$ with $\C P^1 \times \C P^1,$ $\C P^1$ considered as the Riemann sphere $\mathbb{C} \cup \{\infty\}$, let $R_i$ be given as $\{|z| = r_i\}$ and $S_i = \{|w| = \rho_i \}.$ We require that the following {\em non-resonance condition} is satisfied: for all $0 \leq i<j \leq k-1,$ $r_{i}/r_{j} < \rho_{i}/\rho_{j}$ and there are no non-zero integer vectors $\xi = (a_1,\ldots, a_k),\eta = (b_1, \ldots, b_k) \in \Z^k$ such that $\sum a_i \log(r_i) = \sum b_i \log(\rho_i).$ In this situation, by following \cite[Chapter 6, Section 5]{Ahlfors} and \cite{MakSmith-links} it is easy to show the following.

\begin{lma}\label{lma: no Maslov 0}
The non-constant curves $\nu:\Sigma \to M$ from Lemma \ref{lma: connected taut} have non-zero Maslov index: $\mu(\nu) \neq 0.$ The same holds for the restriction of $\nu$ to any connected component of $\Sigma$ with at least two boundary components.
\end{lma}

Indeed, by positivity of intersections and the contrapositive Maslov $0$ assumption, the analysis reduces to there being no simultaneous pair of holomorphic maps $(\nu_1, \nu_2)$ from a compact Riemann surface $C$ with $2\leq l \leq k$ boundary components to $\C$ sending the boundary components to the circles of radii whose logarithms are $\arr{r} = (\log(r_{i_1}),\ldots, \log(r_{i_l}))$ and $\arr{\rho} = (\log(\rho_{i_1}),\ldots, \log(\rho_{i_l})).$ For such a pair of maps, the period matrix $P_l$ of periods of harmonic conjugates of the harmonic measures $w_j$ corresponding to the boundary components of $C$ satisfies: \[P_l \arr{r} = 2\pi \eta \in 2\pi \Z^l,\;\;P_l \arr{\rho} = 2\pi \xi \in 2\pi \Z^l.\] Let us show this for the first projection $\nu_1:C \to \C,$ as the case of the second one is similar. For such a pair of maps, the pull-back $f:C \to \R$ of the harmonic function $\log(|z|)$ to $C$ by $\nu_1$ is a harmonic function on $C$ that is constant on the boundary components of $C.$ Therefore it decomposes as $f = \sum \log(r_j)\, w_j.$ Furthermore $e^f$ extends to the holomorphic function $\nu_1:C \to \C.$ Hence the harmonic conjugate of $f$ must have periods in $2\pi \Z,$ which is equivalent to $P_l \arr{r} \in 2\pi \Z^l.$

Now note that $P_l$ is symmetric and has a one-dimensional kernel corresponding to constants, which is spanned by $(1,\ldots, 1)$. Hence $\xi,\eta \in \Z^l \setminus \{0\},$ and yet \[\brat{\arr{r},\xi} = \frac{1}{2\pi} \brat{\arr{r},P_l \arr{\rho}} =  \frac{1}{2\pi} \brat{P_l \arr{r},\arr{\rho}} =\brat{\eta,\arr{\rho}},\] which cannot hold under the non-resonance condition.

As a consequence, the conditions of Lemma \ref{lma: proportional to 1} are satisfied. What remains is to show that for suitable choices of $\bf b,$ the superpotential $W^{\bf{b}}$ has critical points.



Recall that $I = \{i \in \Z\;|\; 0 \leq i <k\}.$ Consider the coordinates $\{p_i, q_i\}$ on $H^1(\cL_{k,B};\Lambda_0)/H^1(\cL_{k,B};2\pi \sqrt{-1} \Z) \cong (\Lambda_0\setminus \Lambda_{+})^I$ corresponding to the basis $\{e_i, f_i\}$ of $H^1(\cL_{k,B};\C)$ given by the circles $L_i, S_i$ oriented in the direction of the Hamiltonian flow of $z$ on $S^2.$ More precisely, $p_i = \exp(x_i),$ $q_i = \exp(y_i)$ where $x_i, y_i \in \Lambda_0$ are the coordinates on $H^1(\cL_{k,B};\Lambda_0) =  H^1(\cL_{k,B};\C) \otimes \Lambda_0$ corresponding to $\{e_i,f_i\}.$

For $i\in I$ set $n(i)$ to be the number of $j \in I$ such that $j > i$ and $s(i)$ to be the number of $j \in I$ such that $j < i.$ Of course $n(i)+s(i) = k-1$ for all $i \in I.$ 

The part of the superpotential corresponding to smooth disks in $X$ with boundary on $\cl L$ is given by

\[W_{smooth} = T^a \sum_{i \in I} (q_i + q_i^{-1}) + \sum_{i\in I} (T^{n(i)C + B} p_i + T^{s(i)C + B} p_i^{-1}).\]

We let $D_1$ denote the divisor in $X$ corresponding to $D_{(1)} \times M^{k-1}$ for the toric divisor  $D_{(1)} = z^{-1}(\{\pm 1/2\}) \times S^2(2a) \subset M.$

Consider the bulk deformation $\bf b = \beta_{orb}[X_{\gamma}] + \beta D_1$ for $\beta_{orb}, \beta \in \Lambda_+ \cup \{0\}.$ Then the smooth part of the superpotential $W^{\bf b}$ is

\[W^{\bf b}_{smooth} = T^a \sum_{i \in I} (q_i + q_i^{-1}) + e^{\beta} \sum_{i\in I} (T^{n(i)C + B} p_i + T^{s(i)C + B} p_i^{-1}).\]

It remains to calculate the leading order term of the orbifold disk part $W^{\bf b}_{orb}$ of the superpotential. By analyzing the area of each Maslov index $2$ tautological curve ${\nu: (\Sigma,\del \Sigma) \to (M,\mathbb{K}),}$ considered as a curve with boundary on $\bb L',$ with $\Sigma$ having $\leq k-1$ connected components, we see that the area of its homology class $\nu_*[\Sigma,\del \Sigma] \in H_2(M,\mathbb{L}')$ is 

(1) at least $C+a$ if its projection  to the first factor in $M = S^2 \times S^2(2a)$ does not pass through either pole, and \\
(2) at least $B+C$ if its projection to the second factor does not pass through either pole. 

Furthermore, in the first case the minimal area $C+a$ is achieved when $\Sigma$ has $k-1$ connected components, $k-2$ disks and one curve with two boundary components, and $\nu$ is constant on each disk component. Indeed, if in the first case $\Sigma$ has at most $k-2$ connected components, then the area of $\nu$ will be at least $2C+a:$ its projection to the first factor covers at least two annuli contributing at least $2C$ to the area, and by Lemma \ref{lma: no Maslov 0} its projection to the second factor must pass through a pole, contributing at least $a$ to the area. Now on the curve with two boundary components, $\nu$ coincides with the curves analyzed in \cite[Section 3.3]{MakSmith-links}, whereby it contributes to the leading order term of $W^{\bf b}_{orb}$ only if the curve is an annulus.

By choosing translation-invariant trivializations of the tangent bundles of $K_i$ suitably\footnote{We trivialize the tangent bundle of $L_i$ where $i$ is even and of $S_i$ where $i$ is odd along the $S^1$-action on $S^2,$ and that of $L_{i}$ where $i$ is odd and of $S_i$ where $i$ is even along the inverse $S^1$-action. See \cite{MakSmith-links,Cho-holodiskor} for further details on trivializations and orientations.}, and selecting $\beta_{orb} \in \Lambda_{+}$ so that \begin{equation}\label{eq: beta orb} \frac{\beta^2_{orb}}{2} T^{C+a} = T^B \end{equation} $W^{\bf b}_{orb}$ is now given by 


\[ W^{\bf b}_{orb} = T^B \sum_{i \in I\setminus\{k-1\}} \vareps{i}\cdot  p_{i+1}^{-1} p_i(q_{i+1} + q_i^{-1}) + o(T^B)\] for signs $\eps_i \in \{\pm 1 \},$ where $o(T^B)$ denotes higher order terms of valuations strictly greater than $B.$

The superpotential of interest is \[ W^{\bf b} = W^{\bf b}_{smooth} + W^{\bf b}_{orb}.\]

Note that \[W^{\bf b}_{smooth} = T^a \sum_{i \in I} (q_i + q_i^{-1}) + T^{B}(p_{k-1} + p_0^{-1}) + o(T^B)\] since $e^{\beta} = 1 + o(1) \in 1 + \Lambda_{+},$ and hence

\[ W^{\bf b} = T^a \sum_{i \in I} (q_i + q_i^{-1}) + T^B \left(p_{k-1} + p_0^{-1} + \sum_{i \in I\setminus\{k-1\}} \vareps{i}\cdot  p_{i+1}^{-1} p_i(q_{i+1} + q_i^{-1}) \right) + o(T^B)\]

\begin{lma}\label{lma: crit pts}
	For each $\beta \in \Lambda_{+} \cup \{0\}$ the superpotential $W^{\bf b}$ has critical points.
\end{lma}

%
%

\begin{proof}[Proof of Lemma \ref{lma: crit pts}]

We start by finding solutions to the leading order term of the equation $dW^{\bf b} = 0.$

We first consider derivatives in $q_i:$ \begin{equation}\label{eq: q_i derivatives} \del_{q_i} W^{\bf b} = T^a(1 - q_i^{-2}) + o(T^a)\end{equation} where the higher order terms start with valuation at least $B.$ Hence the solutions to the leading order of this equation are $q_i \in \{\pm 1\}.$ Let us choose $q_i = 1$ for all $0 \leq i < k.$

Now we proceed to consider the derivatives in the $p_i,$ for $0< i < k-1:$ \begin{equation}\label{eq: p_i derivatives} \del_{p_i} W^{\bf b} = 2T^B (- \vareps{i-1} p_i^{-2} p_{i-1} + \vareps{i} p_{i+1}^{-1}) +  o(T^B),\end{equation} for $i = 0$ 
\begin{equation}\label{eq: p_0 derivative} \del_{p_0} W^{\bf b} = T^B  (2\vareps{0} p_1^{-1} - p_0^{-2}) + o(T^B).\end{equation} and for $i = k-1$
\begin{equation}\label{eq: p k-1  derivative} \del_{p_{k-1}} W^{\bf b} = T^B  (-  2\vareps{k-2} p_{k-1}^{-2} p_{k-2} + 1) + o(T^B).\end{equation}

The solutions to the leading order equations are hence $p_i \in \C$ satisfying: \begin{align}\label{eq: recursion} p_{i-1}^{-1} p_i^{2} p_{i+1}^{-1} = \sigma_i,\;\; 0<i<k-1,\\ \nonumber p_0^2 p_1^{-1}  =  \sigma_0,\;\; p_{k-2}^{-1} p_{k-1}^{2} =  \sigma_{k-1} \end{align} for certain numbers $\sigma_i \in \R\setminus \{0\}.$ Let us search for solutions $p_i \in \C \setminus \{0\}$ of the form $p_i = \exp(P_i),$ for $P_i \in \C.$ Choosing $C_i$ such that $\sigma_i = \exp(C_i),$ we obtain that it is enough to find $P_i \in \C$ such that for $P = (P_0, \ldots, P_{k-1}),$ $C = (C_0, \ldots, C_{k-1})$ \[ A P = C \] where \[A = \begin{pmatrix}
2 & -1 & 0 & 0 &\dots & 0\\
-1 & 2 & -1 & 0 &\dots & 0\\
0 & & & & & \\
\vdots & & & & &\vdots \\
 & & & & &0 \\
0 & \dots & 0 & -1 & 2 & -1 \\
0 & \dots & 0 & 0 & -1 & 2  \\
\end{pmatrix}\]

is the standard $A_{k}$ Cartan matrix. Since this is a Toeplitz tridiagonal matrix, its eigenvalues are known to be \[ 2\left(1 + \cos\left(\frac{j \pi}{k+1}\right) \right),\; 1\leq j \leq k.\] In particular it is invertible. Hence $p_i = \xi_{i} := \exp(P_i)$ for $P = A^{-1} C$ is a solution to the lowest order term of the critical point equation. 


In summary, $q_i = 1,$ $p_i = \xi_i$ are solutions to the leading order term of $dW^{\bf b} = 0.$ As in \cite{MakSmith-links}, following \cite{FO3-toric-bulk}, we proceed to a solution to the full equation $dW^{\bf b} = 0,$ for $\bf b$ corresponding to $\beta,$ inductively in the $T$-adic valuation. Rewrite the equation $dW^{\bf b} = 0$ in the form \begin{align*}q_i^2 & = 1 + F_{q,i} \\ p_{i-1}^{-1} p_i^{2} p_{i+1}^{-1} &= \sigma_i + F_{p,i}, 0<i<k-1,\\  p_0^2 p_1^{-1} & = \sigma_0 + F_{p,0},\\ p_{k-2}^{-1} p_{k-1}^{2} & = \sigma_{k-1} + F_{p,k-1} \end{align*} where $F_{p,i}, F_{q,i} = o(1)$ are the higher order terms. Since $q_i = 1,$ $p_i = \xi_i$ are solutions to the zero order term of this equation, we search for solutions to the full equation in the form $q_i = \exp(Q_i), p_i =\xi_i \exp(P_i)$ for $Q_i, P_i$ of positive valuation\footnote{Note that for $R = R_0 + R_+ \in \Lambda_0$ of valuation $0,$ where $R_0 \in \C, R_+ \in \Lambda_{+},$ $\exp(R) = \exp(R_0) \cdot \exp(R_+)$ by definition.}. Observe that the matrix $2\id \oplus A,$ where $2\id$ corresponds to the derivative of the left hand side of the equation in the $q_i$ variables and $A$ corresponds to that in the $p_i$ variables, is invertible.  Hence proceeding order by order in $Q_i, P_i$ we obtain the existence of genuine solutions.

Note that as $\beta$ is an element in $\Lambda_+ \cup \{0\},$ it is gapped with respect to a discrete submonoid $G_{\beta} \subset \R_{\geq 0}.$

To make the iterative method work, we look at gapped elements of $\Lambda_0$ with exponents in a suitable discrete submonoid $G \subset \R_{\geq 0}.$ Specifically, we start with the monoid $G_0 = G(\cl L_{k,B},\om,J)$ generated by areas of (orbifold) holomorphic disks with boundary on $\cl L_{k,B}.$ By means of the tautological correspondence it is contained in the submonoid $G(M,\mathbb{L}_{k,B})$ generated by $a, B, C.$ Now let $G_1$ be the submonoid generated by $G_0,$ $G_{\beta}$ and $(B-C-a)/2$ coming from \eqref{eq: beta orb}. Consider the subsets $G^a_1 = \{g \in G_1\;|\; g>a\},$ $G^B_1 = \{g \in G_1\;|\; g>B\}.$ We have $G^a_1-a, G^B_1-B \subset \R_{> 0}.$ The gapped submonoid $G$ that we work with is the one generated by $G^a_1-a$ and $G^B_1 - B.$ Indeed the exponents of the coefficients of all $F_{q,i}, F_{p,i}$ will be contained in it: the coefficients of $F_{q,i}$ are contained in $G^a_1-a$ by Equation \eqref{eq: q_i derivatives}, and the coefficients of $F_{p,i}$ are contained in $G^B_1-B$ by Equations \eqref{eq: p_i derivatives}, \eqref{eq: p_0 derivative}, \eqref{eq: p k-1 derivative}. Since it is gapped we may enumerate its elements as $g_0 = 0 < g_1 < g_2 < \ldots,$ and work inductively assuming a solution modulo $T^{g_j}$ and seeking a solution modulo $T^{g_{j+1}}.$

\end{proof}

\begin{rmk}
In principle the choice $\beta = 0$ is sufficient for our purposes. However, different choices of $\beta$ might lead to different spectral invariants. It would be interesting to explore this dependence further.
\end{rmk}

\begin{proof}[Proof of Theorem \ref{thm: orbifold homology}]
It is a direct consequence of Theorem \ref{thm: crit W} and Lemma \ref{lma: crit pts}.
\end{proof}




\section{Further directions} \label{sec-discussion}

\subsection{Other configurations}\label{subsec-other}  Consider a finite collection $\cC:= \{L_i\}$, $i=1,\dots, k$
of pairwise disjoint embedded circles in $S^2$. Each such collection defines a
graph $\Gamma_\cC$ whose vertices are the connected components of $S^2 \setminus \mathbb{L}$ with $\mathbb{L} = \sqcup L_j$, and a pair of components are joined by an edge if they have a common boundary circle.
Note that $\Gamma_\cC$ is a tree. Additionally, the tree is vertex-weighted: the weight of a vertex
is the area of the corresponding component. We denote the weighting by $w_\cC$. An elementary inductive
argument shows that the pair $(\Gamma_{\cC}, w_\cC)$ determines $\cC$ up to a Hamiltonian isotopy.

For given collection $\cC$, take $a>0$, and denote by  $\cL$ the image of
$\prod (L_i \times S^1_{eq})$ in the $k$-th symmetric product of $S^2 \times S^2(2a)$.

\begin{qtn} \label{qtn-conf} {\rm For which vertex-weighted tree $(\Gamma_{\cC}, w_\cC)$ can one define
		non-trivial Lagrangian estimators coming from the Lagrangian orbifold Floer homology of
		$\cL$, with an appropriate choice of $a>0$?}
\end{qtn}

An obvious necessary condition for non-vanishing of the Floer homology is that for every Hamiltonian diffeomorphism $\phi \in \Ham(S^2)$, the symmetric product of $\phi \times \id \in \Ham(S^2 \times S^2(2a))$  does not displace $\cL$. This is equivalent to the following {\it matching property}: for every $\phi \in \Ham(S^2)$ there exists a permutation $\sigma$ such that
$$\phi(L_i) \cap L_{\sigma(i)} \neq \emptyset \;\; \forall i \in \{1,\dots, d\}\;.$$
It would be interesting (and seems to be not totally trivial) to describe this matching
property in terms of the vertex-weighted tree $(\Gamma_{\cC}, w_\cC)$.

Note that in the present paper we dealt with the case of a linear graph
with the weight $w(v) = B$ if the vertex $v$ has degree $1$, and $w(v) = C$ if the degree of
$v$ is $2$. One readily checks that the inequality $B > C$ is equivalent to the matching property.

Additionally, it would be interesting to explore the analogue of Question \ref{qtn-conf}
on higher genus surfaces.

It is an intriguing and completely open question whether the methods of the present paper
are applicable to Lagrangian configurations on more general symplectic manifolds.
Symplectic toric manifolds provide a promising playground, in which case the simplest Lagrangian configuration is provided by (a suitable modification of) the collection of Bohr-Sommerfeld toric fibers.

Finally, we expect that generalizations of the quantitative methods used in this paper to more complex configurations of Lagrangians would yield further results on quantitative symplectic topology, including the Hofer metric and questions of $C^0$ symplectic topology. We hope to investigate them in a sequel.

\subsection{Next destination: asymptotic cone of $\Ham(S^2)$}
Flats in $\Ham(S^2)$ described in Theorem \ref{thm: main-L} give rise to
infinite-dimensional abelian subgroups in the asymptotic cone (in the sense of Gromov \cite[3.29]{gromov2007metric}) of $\Ham(S^2)$. Recall \cite{calegari2011stable} that the latter is a group equipped with a bi-invariant metric which,
roughly speaking, reflects the large-scale geometry of $(\Ham(S^2), d_{\rm Hofer})$.
For closed surfaces of genus $\geq 2$, the asymptotic cone of the group of Hamiltonian
diffeomorphisms contains a free group with two generators \cite{alvarez2019embeddings,chor2020eggbeater}.
The construction is based on a chaotic dynamical system called {\it the eggbeater map}
(see \cite{PolShe} for symplectic aspects of this map). For the torus and the sphere, existence of a free non-abelian subgroup in the asymptotic cone  is still unknown.
Furthermore, while there is a hope that a suitable modification of eggbeaters
could work in the case of torus, in the case of the sphere  egg-beaters fail to induce a non-abelian subgroup in the asymptotic cone (an observation of Michael Khanevsky).
It would be interesting to explore the algebraic and geometric structure of the asymptotic cone of $\Ham(S^2)$. As a first step, it would be natural to explore asymptotic growth (in the sense of Hofer's metric) of subgroups generated by a finite number of pair-wise ``highly non-commuting" flats.

\subsection{Comparison with periodic Floer Homology?}
Instead of pulling back spectral estimators $c^0_{k,B}$ on $S^2$ from
$S^2 \times S^2(2a)$ (see Theorem \ref{thm: invariants sphere} above), we could have run our construction omitting the factor $S^2(2a)$.
In other words, we could have worked directly with orbifold Lagrangian spectral invariants
on the symmetric products on $S^2$. In this way the spectral estimators become
well-defined when $B=C = 1/(k+1)$, which is the limiting case for the assumptions
of Theorem \ref{thm: invariants sphere}.  Interestingly enough, in this limiting case
for certain radially symmetric Hamiltonians, our invariants agree with the ones constructed in
\cite{CGHS2} by means of periodic Floer homology. It would be interesting to find a conceptual
explanation of this coincidence.

\color{black}

\section*{Acknowledgements}
We thank Dan Cristofaro-Gardiner, Vincent Humili\`{e}re, and Sobhan Seyfaddini for very useful communications regarding their work \cite{CGHS2}. While our original approach was based
on the study of Lagrangian configurations of 2 circles (which suffices for proving the results
on flats in Hofer's geometry), our interest in configurations of $k \geq 3$ circles was in part triggered by their work.

We thank Mohammed Abouzaid, Michael Brandenbursky, Cheuk Yu Mak, and Ivan Smith for useful discussions, Vitaly Bergelson for bringing \cite[Theorem 1.2]{Bergelson} to our attention, and Gerhard Knieper for a stimulating question leading to Corollary \ref{cor: Banach}.

LP was partially supported by the Israel Science Foundation grant 1102/20.

ES was partially supported by an NSERC Discovery Grant, by the Fonds de recherche du Qu\'{e}bec - Nature et technologies, and by the Fondation Courtois.

\bibliographystyle{abbrv}
\bibliography{bibliographyS2-arXiv6-final}
\end{document}